\documentclass[twoside,11pt]{article}

\usepackage{blindtext}
\usepackage{multirow,booktabs}
\usepackage{enumerate}
\usepackage[dvipsnames]{xcolor}
\usepackage{fullpage}
\usepackage{lipsum,stackengine}
\setstackEOL{\\}
\usepackage{lastpage}
\usepackage{soul}
\usepackage{enumitem}
\usepackage[ruled,vlined]{algorithm2e}
\usepackage{fancyhdr}
\usepackage{mathrsfs}
\usepackage{stackrel}
\usepackage{wrapfig}
\usepackage{setspace}
\usepackage{calc}
\usepackage{pdfpages}
\usepackage{multicol}
\usepackage{cancel}
\usepackage[retainorgcmds]{IEEEtrantools}
\usepackage[margin=3cm]{geometry}
\usepackage{amsmath}
\usepackage{macros}

\usepackage{empheq}
\usepackage{framed}
\usepackage[most]{tcolorbox}
\usepackage{xcolor}
\usepackage{tikz}
\usepackage{forest}
\usepackage[font=small,labelfont=bf,margin=\parindent,tableposition=top]{caption}

\colorlet{shadecolor}{orange!15}
\usepackage[preprint]{jmlr2e}

\usepackage{lastpage}
\jmlrheading{23}{2022}{1-\pageref{LastPage}}{1/21; Revised 5/22}{9/22}{21-0000}{Author One and Author Two}

\usepackage[font=small,labelfont=bf,margin=\parindent,tableposition=top]{caption}
\usepackage{subcaption}  

\ShortHeadings{Disciplined Geodesically Convex Programming}{Cheng and Dixit et al.}
\firstpageno{1}

\begin{document}

\title{Disciplined Geodesically Convex Programming}

\author{\name Andrew N. Cheng$^*$ {\email andrewcheng@g.harvard.edu \\
       \addr Harvard University\\
       Cambridge, MA 02138, USA}
       \AND
       \name Vaibhav Dixit$^*$ {\email vkdixit@mit.edu \\
       \addr CSAIL, MIT\\
       Cambridge, MA 02139, USA}
       \AND
       \name Melanie Weber \email mweber@seas.harvard.edu \\
       \addr Harvard University\\
       Cambridge, MA 02138, USA
       }

\editor{My editor}

\maketitle
\def\thefootnote{*}\footnotetext{Equal contribution. Co-first authors listed alphabetical.}

\begin{abstract}
Convex programming plays a fundamental role in machine learning, data science, and engineering. Testing convexity structure in nonlinear programs relies on verifying the convexity of objectives and constraints. \citet{grant2006disciplined} introduced a framework, \emph{Disciplined Convex Programming} (DCP), for automating this verification task for a wide range of convex functions that can be decomposed into basic convex functions (atoms) using convexity-preserving compositions and transformations (rules).
Here, we extend this framework to functions defined on manifolds with non-positive curvature (Hadamard manifolds) by introducing \emph{Disciplined Geodesically Convex Programming} (DGCP). In particular, this allows for verifying a broader range of convexity notions. For instance, many notable instances of statistical estimators and matrix-valued (sub)routines in machine learning applications are Euclidean non-convex, but exhibit \emph{geodesic} convexity through a more general Riemannian lens. To define the DGCP framework, we determine convexity-preserving compositions and transformations for geodesically convex functions on general Hadamard manifolds, as well as for the special case of symmetric positive definite matrices, a common setting in matrix-valued optimization. For the latter, we also define a basic set of atoms. Our paper is accompanied by a Julia package \textsl{SymbolicAnalysis.jl}, which provides functionality for testing and certifying DGCP-compliant expressions. Our library interfaces with manifold optimization software, which allows for directly solving verified geodesically convex programs.
\end{abstract}


\section{Introduction}
Nonlinear programming, which involves optimization tasks with nonlinear objectives and/or nonlinear constraints, plays a fundamental role in data science, machine learning, engineering, operations research, and economics. Classically, nonlinear programs are solved with Euclidean optimization methods, whose design and mathematical analysis has been the subject of decades of research. Structured nonlinear programs can often be solved more efficiently with specialized methods. This has given rise to a wide range of algorithms for solving special classes of nonlinear programs that leverage special structure in the programs’s objective and constraints. \emph{Convex programming} involves nonlinear programs with Euclidean convex objectives and constraints, which gives rise to efficient algorithms with global optimality certificates. While convex programming has a wide range of applications, there are many notable instances in data science and machine learning that do not fit into this restrictive setting. This includes the computation of several important statistical estimators, such as Tyler's and related M-estimators~\citep{Tyler1987,wiesel2012geodesic,ollila2014regularized}, optimistic likelihood estimation~\citep{nguyen2019calculating}, and certain Wasserstein bounds on entropy~\citep{courtade2017wasserstein}. Furthermore, a number of matrix-valued (sub-) routines that arise in machine learning approaches fall into this setting, including robust subspace recovery~\citep{zhang2016robust}, matrix barycenter problems~\citep{Bhatia1997_matrixanalysis}, and learning Determinantal Point Processes (DPPs)~\citep{mariet2015fixed}. However, a closer analysis of the properties of these nonlinear programs can reveal “hidden” convexity structure, when viewed through a geometric lens: While their objectives and/or constraints may be Euclidean non-convex, they are convex with respect to a different Riemannian metric.

A notable setting where such convexity structure arises are optimization tasks on symmetric positive definite matrices. We can endow this space either with a Euclidean metric or with the affine-invariant Riemannian metric, in which case they form a Cartan-Hadamard manifold, i.e., a manifold of non-positive sectional curvature. The sample applications listed above exhibit convexity in the Riemannian setting only. In practice, if we can reliably identify under which metric a given program exhibits such \emph{geodesic convexity}, we can leverage efficient convex optimization tools with global optimality guarantees. This observation motivates the need for tools that can effectively test and verify the convexity of the objective and constraints of nonlinear programs under generalized metrics. While this can be done “by hand” via mathematical analysis, the development of computational tools that automate this procedure and that can be integrated into numerical software would ensure broad applicability. In the Euclidean setting, \emph{Disciplined Convex Programming}~\citep{grant2006disciplined} (short: \emph{DCP}) has been introduced as a framework for automating the verification of convexity. It decomposes the objective function or a functional description of the constraints into basic functions that are known to be convex (so-called \emph{atoms}) using convexity-preserving compositions and transformations (known as \emph{rules}). The CVX library~\citep{diamond2016cvxpy} implements this framework and provides an interface with numerical convex optimization tools. More recently, the DCP framework has been extended to log-log convex~\citep{dgp} and quasi-convex~\citep{dqp} programs. However, to the best of our knowledge, no extensions of this framework to the geodesically convex setting have been considered.

In this work, we introduce a generalization of the DCP framework that leverages the intrinsic geometry of the manifold to test convexity. The extension to the \emph{geodesically convex} setting encompasses Euclidean convex programming, as well as programs with objectives and constraints that are convex with respect to more general Riemannian metrics (\emph{Disciplined Geodesically Convex Programming}, short: \emph{DGCP}). 
We provide a structured overview of geodesic convexity-preserving compositions and transformations of functions defined on Cartan-Hadamard manifolds, which serve as a foundational set of rules in our DGCP framework.
Focusing on optimization tasks defined on symmetric positive definite matrices, we define additional rules, as well as a basic set of geodesically convex atoms that allow for testing and certifying the convexity of many classical matrix-valued optimization tasks. This includes in particular statistical estimators and many of the aforementioned subroutines in machine learning and data analysis methods. We further present an accompanying open-source package, \textsl{SymbolicAnalysis.jl} \footnote{\url{https://github.com/Vaibhavdixit02/SymbolicAnalysis.jl}}, which implements DGCP, and illustrate its usage on several classical examples.

\paragraph{Related Work.}
Convex programming has been a major area of applied mathematics research for many decades~\citep{Boyd_Vandenberghe_2004}. Extensions of classical convex optimization algorithms to manifold-valued tasks have been studied extensively, resulting in generalized algorithms for convex~\citep{udriste1994convex,bacak2014convex,zhang2016first}, nonconvex~\citep{boumal2019global}, stochastic~\citep{bonnabel2013stochastic,zhang2016riemannian,weber2021projection}, constrained~\citep{weber2022riemannian,weber2021projection,bergmann2019intrinsic,bergmann2022first}, and min-max optimization problems~\citep{martinez2023accelerated,jordan2022first}, among others.
Numerical software for solving geometric optimization problems has been developed in several languages~\citep{manopt,pymanopt,manoptjl,roptlib}. Disciplined Convex Programming for testing and certifying the Euclidean convexity of nonlinear programs has been developed by~\citet{grant2006disciplined} and made available in the CVX library~\citep{diamond2016cvxpy}. More recently, extensions to quasi-convex programs (\emph{Disciplined Quasi-Convex Programming}~\citep{dqp}) and log-log convex programs (\emph{Disciplined Geometric Programming}~\citep{dgp}) have been integrated into CVX. We note that, in the latter, the term ``geometric'' is used in a different context than in our work: Log-log convexity is a Euclidean concept that evaluates convexity under a specific transformation. In contrast, the notion of geodesic convexity considers the geometry of the domain explicitly. To the best of our knowledge, no extensions of disciplined programming to the geodesically convex setting have been introduced in the prior literature.

\newpage
\paragraph{Summary of contributions.}
The main contributions of this work are as follows:
\begin{enumerate}
    \item We introduce \emph{Disciplined Geodesically Convex Programming}, a generalization of the Disciplined Convex Programming framework, which allows for testing and certifying the geodesic convexity of nonlinear programs on geometric domains.
    \item Following an analysis of the algebraic structure of geodesically convex funcations, we define convexity-preserving compositions and transformations for geodesically convex functions on Cartan-Hadamard manifolds, as well as for the special case of symmetric positive definite matrices, for which we also define a foundational set of atoms.
    \item For the special case of symmetric positive definite matrices, we present an implementation of this framework in the Julia language ~\citep{bezanson2017julia}. Our open-source package, \textsl{SymbolicAnalysis.jl} allows for verifying DGCP-compliant convexity structure and interfaces with  manifold optimization software, which allows for directly solving verified programs.
\end{enumerate}

\section{Background and Notation}
In this section, we introduce notation and review standard notions of Riemannian geometry and optimization. For a comprehensive overview see~\citep{boumal2020introduction,bacak2014convex}.

\subsection{Riemannian Geometry}\label{sec:Riemannian_Geometry}
A \textit{manifold} $\mathcal{M}$ is a topological space that has a local Euclidean structure. Every $x \in \mathcal{M}$ has an associated \textit{tangent space} $\mathcal{T}_x \mathcal{M}$, which consists of the tangent vectors of $\mathcal{M}$ at $x$. We restrict our attention to \textit{Riemannian manifolds}, which are endowed with a smoothly varying inner product $\langle u, v \rangle_x$ defined on $\mathcal{T}_x \mathcal{M}$ for each $x \in \mathcal{M}$. More specifically, we consider a special class of Riemannian manifolds called the \emph{Cartan-Hadamard manifolds} These are manifolds with non-positive sectional curvature. Importantly, the class of Cartan-Hadamard manifolds is appealing for optimization due to properties such as \textit{unique length-minimizing geodesics} and amenablility to \textit{geodesic convexity analysis}~\citep{bacak2014convex}. 

\paragraph{Symmetric Positive Definite Manifold.}
A special instance considered in this paper is the manifold of symmetric positive definite matrices, denoted as $\pd$, which we encounter frequently in matrix-valued optimization. Formally, it is given by the set of $d \times d$ real symmetric square matrices with strictly positive eigenvalues, i.e., 
\begin{equation*}
    \pd := \{ X \in \real^{d\times d}: X^T=X, \; X \succ 0 \}
\end{equation*}

Endowing $\pd$ with different inner product structures gives rise to different Riemannian lenses on $\pd$. We recover a Euclidean structure if we endow $\pd$ with 
\[
\langle A, B \rangle = \tr(A^\top B) \qquad \forall A,B \in \pd.
\]
We can induce a \textit{non-flat} Riemannian structure of $\pd$ by endowing $\pd$ with the canonical \textit{affine invariant} inner product, 
\[\langle A, B\rangle_X=\operatorname{tr}\left(X^{-1} A X^{-1} B\right) \quad X \in \mathbb{P}_d, \; A, B \in \mathcal{T}_X\left(\mathbb{P}_d\right)=\mathbb{S}_d,\]
where the tangent space $\mathcal{T}_X\left(\mathbb{P}_d\right)=\mathbb{S}_d$ is the space of $d\times d$ real symmetric matrices. On $\pd$, given any matrices $A,B \in \pd$, the unique geodesic connecting $A$ to $B$ has the explicit parametrization
\begin{equation}\label{eq:intro_gcvx_def}
    \gamma(t)=A^{1 / 2}\left(A^{-1 / 2} B A^{-1 / 2}\right)^t A^{1 / 2}, \quad 0 \leq t \leq 1 \; .
\end{equation}
The affine-invariant structure on $\pd$ gives rise to the following \textit{Riemannian distance} on $\pd$, 
\[
\delta_{R}(A, B)=\left\|\log A^{-1 / 2} B A^{-1 / 2}\right\|_F \; ,
\]
which corresponds to the length of the geodesic connecting $A$ and $B$. It is geodesically convex, since $\pd$ is Cartan-Hadamard~\citep{bacak2014convex, bhatia07positivedefinitematrices}.

\paragraph{Lorentz Model.} To show the versatility of our framework, we consider another special instance of the Cartan-Hadamard manifold, namely the $d$-dimensional Lorentz model $(\mathbb{H}^d, d_\mathcal{L})$. In the Lorentz model, the \emph{non-flat} Riemannian structure is induced by the \textit{Lorentzian inner product} $\langle \cdot, \cdot \rangle_\mathcal{L}: \real^{d+1} \to \real$ defined by 
\[
\langle x, y \rangle_\mathcal{L} = x_1 y_1 + \cdots + x_d y_d - x_{d+1}y_{d+1}, \qquad x,y \in \real^{d+1}.
\]
We may also write 
\[
\langle x, y \rangle_\mathcal{L} = x^\top J y \qquad \text{where} \qquad J := \diag(1, \ldots, 1, -1).
\]
Then the $d$-dimensional Lorentz model $\mathbb{H}^d$ and its tangent space at a point $p \in \mathbb{H}^d$ is defined as 
\[
\begin{aligned}
\mathbb{H}^d & :=\left\{p \in \mathbb{R}^{d+1}:\langle p, p\rangle=-1, p^{n+1}>0\right\}, \\
T_p \mathbb{H}^d & :=\left\{v \in \mathbb{R}^{d+1}:\langle p, v\rangle=0\right\},
\end{aligned}
\]
respectively. The Lorentzian structure gives rise to the following \textit{Riemannian distance} on $\mathbb{H}^d$ 
\[
d_\mathcal{L}(p,q) := \operatorname{arcosh}(-\langle p, q \rangle_\mathcal{L}). 
\]
 Given any two points $p,q \in \mathbb{H}_d$, the unique geodesic connecting $p$ and $q$ in $\mathbb{H}_d$ has the explicit parametrization 

\[
\gamma(t)=\left(\cosh t+\frac{\langle p, q\rangle \sinh t}{\sqrt{\langle p, q\rangle^2-1}}\right) p+\frac{\sinh t}{\sqrt{\langle p, q\rangle^2-1}} q, \quad \forall t \in[0, d(p, q)].
\]

\subsection{Geodesic Convexity of Functions and Sets}
Many classical results from Euclidean convex analysis can be extended to Cartan-Hadamard manifolds. Below, we introduce the analogous notions of convexity of sets and functions in the Riemannian setting. The definitions in this section hold for Riemannian manifolds
$\mathcal{M}$. We only consider functions that are continuous. 
\begin{definition}[Geodesic convexity of Sets]\label{def:g-convex-s}
A set $S \subseteq \mathcal{M}$ is \emph{geodesically convex} (short: g-convex) if for any two points $x,y \in \mathcal{M}$, there exists a geodesic $\gamma:[0,1] \to  \mathcal{M}$ such that $\gamma(0) = x$ and $\gamma(1) = y$ and the image satisfies $\gamma([0,1]) \subseteq S$.\footnote{For geodesically convex sets on Cartan-Hadamard manifolds, any such geodesic segment is unique.}
\end{definition}
\begin{definition}[Geodesic convexity of Functions]
\label{def:g-convex-f}
    We say that $\phi: S \to \real$ is a \emph{geodesically convex function} (short: g-convex) if $S \subseteq \mathcal{M}$ is geodesically convex and $f \circ \gamma :[0,1] \to \real$ is (Euclidean) convex for each geodesic segment $\gamma :[0,1] \to \pd$ whose image is in $S$ with $\gamma(0) \neq \gamma(1)$.
\end{definition}

As we will see in Section~\ref{sec:rules}, many of the operations that preserve Euclidean convexity extend to the geodesically convex setting. In Appendix~\ref{app:g_cvx_different_metrics}, we illustrate how the convexity of functions depends naturally on the geometry of the Riemannian manifold.

\subsection{Riemannian optimization software}

A widely used library for manifold optimization is the \textsl{Manopt} toolbox~\citep{manopt}, a MATLAB-based software designed to facilitate the experimentation with and application of Riemannian optimization algorithms. \textsl{Manopt} simplifies handling complex optimization tasks by providing user-friendly and well-documented implementations of various state-of-the-art algorithms. It separates the manifolds, solvers, and problem descriptions, allowing easy experimentation with different combinations. 
In addition to the MATLAB version, a Python implementation has been made available (\textsl{PyManopt}~\citep{pymanopt}).

In the Julia programming language, \textsl{Manopt.jl}~\citep{manoptjl} offers a comprehensive framework for optimization on Riemannian manifolds. It utilizes \textsl{Manifolds.jl} ~\citep{axen2023manifolds} for efficient implementations of manifolds like the Euclidean, hyperbolic, and spherical spaces, the Stiefel manifold, the Grassmannian, and the positive definite matrices, among others, which also includes an efficient implementation of important primitives on these manifolds like geodesics, exponential and logarithmic maps, parallel transport, etc. 
Additionally, there are other software packages such as \textsl{ROPTLIB} for C++~\citep{roptlib}, which manifold optimization tools in other languages.

\section{Disciplined Geodesically Convex Programming}
In this section we introduce the \emph{Disciplined Geodesically Convex Programming} framework (short: \emph{DGCP}). We discuss the relationship to other classes of convex programming, as well as the essential building blocks of the framework.

\subsection{Taxonomy of Convex Programming}
\begin{figure}[t]
    \centering
\includegraphics[width=0.6\textwidth]{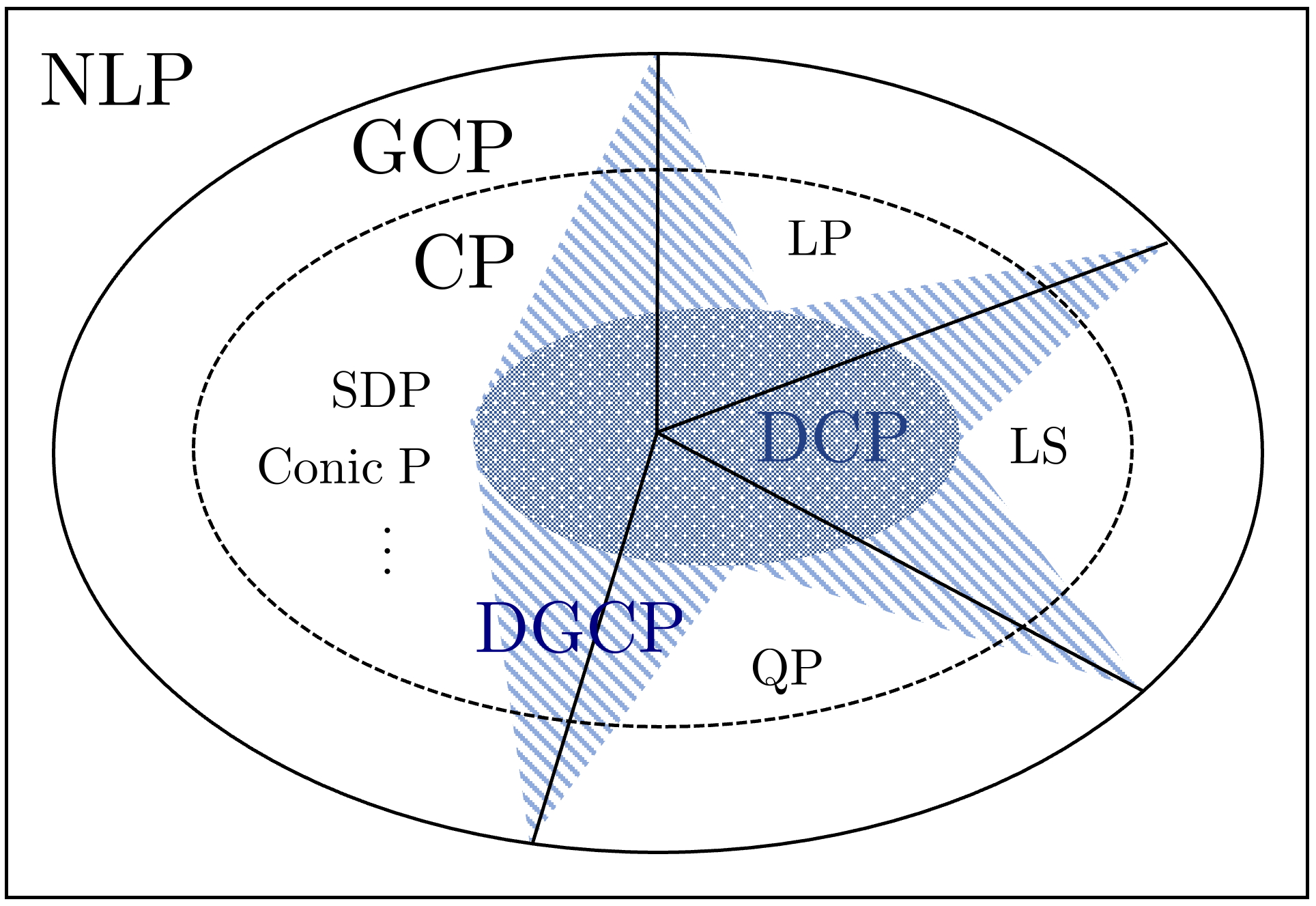}
    \caption{\textbf{Taxonomy of Convex Programming.}
     The diagram shows the relationship of GCP, CP and their subclasses (e.g., SDP, LP, QS etc.). DGCP (blue shaded) has non-empty intersections with GCP, CP and their subclasses and contains DCP (gray shaded) as a special case.
    }

    \label{fig:taxonomy}
\end{figure}
We consider \emph{nonlinear programs} (NLP) of the form
\begin{align}\label{eq:nlp}
    \min_{x \in \R^{n \times n}} \quad &f(x) \\
    {\rm subject \; to} \quad &g_i(x) \leq 0, \; i=1,\dots,m \nonumber \\
    &h_j(x) = 0, \; j=1,\dots,n \; ,\nonumber
\end{align}
which are defined by an objective function $f: \R^{n \times n} \rightarrow \R$ and a set of inequality $\{g_i\}_{i \in [m]}$ and equality constraints $\{h_j\}_{j \in [n]}$ (where $[n]:=1, \dots, n$). 

\paragraph{Convex Programming.} \emph{Convex programs} (CP) are a class of NLPs, in which both the objective and the constraints are convex. Classically, ``convexity'' refers to Euclidean convexity. Here, we consider the more general class of \emph{geodesically convex programs} (GCP), which require that the objective and constraints are geodesically convex under \emph{some} Riemannian metric, but not necessary the Euclidean metric. This extends the framework to optimization tasks where the objective and/ or constraints are geodesically convex under some non-Euclidean, Riemannian metric.
Hence, CP $\subset$ GCP. CP encompasses Linear Programming (LP), Quadratic Programming (QP), Least Squares (LS) problems, as well as a number of optimization problems with special structure, such as semidefinite programs (SDP) and conic programs (Conic P). 

From an algorithmic perspective, (geodesic) convexity enables certificates of \emph{global optimality}, in that local optima are guaranteed to be global optima. Since local optimality can be verified, e.g., via KKT conditions, this allows for  global convergence guarantees from any initialization in practice -- a highly desirable property. 
Hence, CP and its subclasses have been extensively studied in the Euclidean optimization literature. More recently, GCP~\citep{udriste1994convex,bacak2014convex,boumal2020introduction,absil_interpolation}, as well as generalizations of the CP subclasses to the geodesic setting have been studied~\citep{sra2015conic}.

\paragraph{Disciplined Programming.}
Due to the algorithmic benefits discussed above, identifying and verifying CP is of great interest in practise. Aside from formally proving convexity certificates (i.e., verifying Def.~\ref{def:g-convex-s} for objective functions and Def.~\ref{def:g-convex-f} for the feasible region), one can also leverage the algebraic structure of convex functions to discover convexity in objectives and constraints. Specifically, many transformations or compositions of convex functions yield convex functions. The idea of \emph{Disciplined Convex Programming} (DCP)~\citep{grant2006disciplined} is to define a set of \emph{atoms} and \emph{rules} to verify convexity properties. Atoms are functions and sets whose properties in terms of convexity and monotonicity are known. Rules encode fundamental principles from convex analysis on transformations and compositions that preserve or induce convexity in functions or sets. Together, they form a modular framework for verifying convexity in functions and sets that can be decomposed into atoms using any combination of rules. In principle, any function that is not verifiable using existing atoms and rules could be added as a new atom, which would allow for creating a library of rules and atoms that could verify the convexity of any CP. However, in practise, DCP libraries are limited to a set of core atoms and rules that allow for verifying commonly encountered mathematical programs. Hence, generally DCP $\subset$ CP.

In this work, we extend the idea of disciplined programming to the geodesically convex setting. We design a library of geodesically convex atoms (sec.~\ref{sec:atoms}) and rules for preserving or inducing geodesic convexity in functions and sets (sec.~\ref{sec:rules}). The resulting library, termed \emph{Disciplined Geodesically Convex Programming} (DGCP) allows for verifying a larger subset of CP, as well as a subset of programs that are in GCP, but not in CP. Thus DGCP $\subset$ GCP. A schematic overview of the taxonomy of the different classes of convex programs can be found in Figure~\ref{fig:taxonomy}.


\subsection{General Cartan-Hadamard Manifolds}
In this work we focus on developing a disciplined programming framework for Cartan-Hadamard manifolds.
Cartan-Hadamard manifolds are manifolds of non-positive sectional curvature with the property that every pair of points can be connected by a unique geodesic that is distance-minimizing with respect to its Riemannian metric. This is a key property in generalizing tools from Euclidean convex analysis to the Riemannian setting (e.g., geodesic convexity) in a global sense. In contrast, such tools cannot be as readily imported to manifolds with positive sectional curvatures. For example, spheres do not admit globally geodesically convex functions beyond the constant function and key operations such as intersections of sets fail to preserve geodesic convexity on spheres. In addition, Cartan-Hadamard manifolds arise in many data science and machine learning application; hence, a disciplined programming framework for this class of manifolds has a wide range of potential applications.

\subsubsection{Rules for Cartan-Hadamard Manifolds}
In this section, we present operations that are \emph{DGCP-compliant for general Cartan-Hadamard manifolds}, i.e., operations that preserve geodesic convexity of functions.
After introducing a general set of DGCP-compliant rules, we focus on two instances of Cartan-Hadamard manifolds: the symmetric positive definite manifold (Section~\ref{sec:rules}) and the Lorentz model (Section~\ref{sec:Lorentz_model}). For each instance, we provide an additional DGCP-compliant rules that are specific to their geometry.
We defer all proofs to Appendix~\ref{app:gcvx_rules}.

\begin{prop}\label{prop:coniccomb_pwmax}
    Let $(\mathcal{M}, d)$ be a Cartan-Hadamard manifold. Suppose $S \subseteq \mathcal{M}$ is a g-convex subset. Furthermore, suppose $f_i: S \to \real$ are g-convex for $i= 1, \ldots, n$. Then the following functions are also g-convex.
     \begin{enumerate}
        \item $X \mapsto \max_{i \in \{1,\ldots, n\}}f_i(X)$
        \item $X \mapsto \sum_{i=1}^n \alpha_i f_i(X) $ for $\alpha_1, \ldots, \alpha_n  \geq 0$.
    \end{enumerate}
\end{prop}

{
\begin{remark}
    In the setting of Cartan-Hadamard manifolds, property 1 of Proposition~\ref{prop:coniccomb_pwmax} can be generalized to an arbitrary collection of g-convex sets. That is, for an arbitrary collection of g-convex functions $\{f_i\}_{i \in\mathcal{I}}$, indexed by $\mathcal{I}$, the map $X \mapsto \sup_{i \in \mathcal{I}}f_i(X)$ is g-convex. This follows from the fact that a function $f$ is g-convex if and only if its epigraph is g-convex \citep{bacak2014convex} and the fact that the epigraph of the supremum of a collection of functions is the intersection of the epigraphs of each function in such a collection. Finally, the intersection of g-convex sets is g-convex for Cartan-Hadamard manifolds (see, e.g.,~\citep{boumal2020introduction}).
    Moreover, property 2 of Proposition~\ref{prop:coniccomb_pwmax} can easily be generalized to a countable conic sum of g-convex functions. 
\end{remark}
}

The following rule gives a convexity guarantee for compositions of Euclidean and g-convex functions. 
\begin{prop}\label{prop:ecvx_composition}
     Let $(\mathcal{M}, d)$ be a Cartan-Hadamard manifold and $S \subset \mathcal{M}$ g-convex. Suppose $f: S \rightarrow \mathbb{R}$ is g-convex. If $h: \mathbb{R} \rightarrow \mathbb{R}$ is non-decreasing and Euclidean convex then $h \circ f: S \to \real$ is g-convex.
\end{prop}

We also have the following analogous results.
\begin{corollary}[Scalar Composition Rules]
        \begin{enumerate}
            \item[]
           \item Let $f: S \rightarrow \mathbb{R}$ be geodesically concave. If $h: \mathbb{R} \rightarrow \mathbb{R}$ is non-increasing and convex, then $h \circ f$ is geodesically convex on $S$.
            \item Let $f: S \rightarrow \mathbb{R}$ be geodesically concave. If $h: \mathbb{R} \rightarrow \mathbb{R}$ is non-decreasing and concave, then $h \circ f$ is geodesically concave on $S$.
            \item Let $f: S \rightarrow \mathbb{R}$ be geodesically convex. If $h: \mathbb{R} \rightarrow \mathbb{R}$ is non-increasing and concave, then $h \circ f$ is geodesically convex on $S$.
        \end{enumerate}
    \end{corollary}

\begin{example}
    If $f:S \to \real$ is g-convex with respect to the canonical Riemannian metric then $\exp f(x)$ is g-convex and $- \log (-f(x))$ is g-convex on $\{x : f(x) < 0 \}$. If $f$ is non-negative and $p \geq 1$ then $f(x)^p$ is g-convex.
\end{example}

\subsubsection{Atoms for Cartan-Hadamard Manifolds}
In this section, we present geodesically convex atoms on a Cartan-Hadamard manifold $(\mathcal{M}, d).$ In the next section, we present atoms specific to the geometry of the symmetric positive definite manifold and the Lorentz model. 

\begin{example}[\cite{bacak2014convex}]
    Let $(\mathcal{M}, d)$ be a Cartan-Hadamard manifold. The following functions $f: \mathcal{M} \to \real$ are geodesically convex.
    \begin{enumerate}
        \item Let $y \in \mathcal{M}$ then the intrinsic distance to $y$ given by $f(x) = d(x,y)$ is geodesically convex. More generally,  
        \[
        f(x) \defas d^p(x,y) 
        \]
        is geodesically convex for $p \geq 1.$ Furthermore, let $\{x_i\}_{i=1}^n \subseteq \mathcal{M}$ and $w_1, \ldots, w_n >0$ such that $\sum_{i=1}^n w_i  = 1$. Then 
        \begin{equation}
        f(x) = \sum_{i=1}^n w_i d^p(x, x_i)    
        \end{equation}
        is geodesically convex for $p \geq 1$.
        \item Let $F:\mathcal{M} \to \mathcal{M}$ be an isometry. Then the function 
        \[
        f_F(x) \defas d(x, Fx) 
        \]
        is geodesically convex.
        
    \end{enumerate}
\end{example}

\subsection{Manifold of Symmetric Positive Definite matrices}\label{sec:rules}

Our DGCP framework can be specialized to any Cartan-Hadamard manifold. In addition to the general rules introduced in the previous section, additional sets of g-convexity preserving rules may be defined that arise from a manifold's specific geometry. In this section, we illustrate this for the special case of symmetric positive definite matrices, i.e., by setting $\mathcal{M} = \pd$, and  $d = \delta_R(A, B):=\left\|\log A^{-1 / 2} B A^{-1 / 2}\right\|_F$. Below we introduce a set of g-convexity preserving rules and geodesically convex atoms that are inherent to this particular geometry.

The Löwner order introduces a partial order relation on the symmetric positive definite matrices which will be used to establish g-convexity results.
\begin{definition}[Löwner Order]\label{def:loewner_order}
    For $A ,B \in \pd$ we write $A \succ B$ when $A - B \in \pd$. Similarly, we write $A \succeq B$ whenever $A -B$ is symmetric positive semi-definite.
\end{definition}
We say a function $f: \pd \rightarrow \R$ is \textit{increasing} if $f(A) \succeq f(B)$ whenever $A \succeq B$.

\begin{definition}[Positive Linear Map]
    A linear map $\Phi:\mathbb{P}_d \to \mathbb{P}_m$ is \textit{positive} when $\Phi(A) \succeq 0$ for all $A \in \pd$. We say that $\Phi$ is \textit{strictly positive} when $A \succ 0$ implies that $\Phi(A) \succ 0$.
\end{definition}

\subsubsection{Symmetric Positive Definite Manifold Rules}
The following proposition gives a g-convexity guarantee for compositions of strictly positive linear maps.
\begin{prop}[Proposition 5.8~\citep{Vishnoi2018GeodesicCO}]\label{prop:strict_positive_linear}
    Let again $\Phi(X)$ be a strictly positive linear operator from $\mathbb{P}_d$ to $\mathbb{P}_m$. Then $\Phi(X)$ is g-convex with respect to the Löwner order on $\mathbb{P}_m$ over $\mathbb{P}_d$ with respect to the canonical Riemannian inner product $g_X(U, V):=$ $\operatorname{tr}\left[X^{-1} U X^{-1} V\right]$. In other words, for any geodesic $\gamma:[0,1] \rightarrow \mathbb{P}_d$ we have that
$$
\Phi(\gamma(t)) \preceq(1-t) \Phi(\gamma(0))+t \Phi(\gamma(1)) \quad \forall t \in[0,1] \; .
$$
\end{prop}
Consequently, the following maps are g-convex in this setting:
\begin{example}[Strictly Positive Linear Operators]
    Let $Y \in \pd$ fixed. Applying Proposition~\ref{prop:strict_positive_linear} the following maps are g-convex w.r.t the canonical Riemannian metric on $\pd$:
    \begin{enumerate}
    \item $X \mapsto \tr(X)$ 
    \item $X \mapsto Y^\top X Y$ for $Y \in \real^{d \times k}$
    \item $X \mapsto \operatorname{Diag}(X) := \sum_{j}X_{jj}E_{jj}$, where $E_{jj}$ is the $d\times d$ matrix with 1 in the $(j,j)$-th element and 0 everywhere else.
    \item Let $M \succeq 0$ and $M$ has no zero rows.  The function $\Phi(X) = M \odot X$ where $\odot$ denotes the Hadamard product is a strictly positive linear operator and hence g-convex.
    \end{enumerate}
\end{example}
Moreover, the following proposition guarantees that the composition of positive linear maps with $\log \det(\cdot)$ is g-convex.
\begin{prop}[Proposition 5.9~\citep{Vishnoi2018GeodesicCO}]\label{prop:logdet_gcvx} Let $\Phi(X):\pd \to \mathbb{P}_m$ be a strictly positive linear operator. Then, $\log \operatorname{det}(\Phi(X))$ is g-convex on $\pd$ with respect to the metric $g_X(U, V):=\operatorname{tr}\left[X^{-1} U X^{-1} V\right]$.
\end{prop}
\begin{prop}\label{lemma:inverse_gcvx}
    Let $f: \pd \to \real$ be g-convex.
    Then $g(X) = f(X^{-1})$ is also g-convex.
\end{prop}

\begin{example}
    Applying Proposition~\ref{prop:logdet_gcvx} and Lemma~\ref{lemma:inverse_gcvx} the following maps are g-convex with respect to the canonical Riemannian metric.
    \begin{enumerate}
        \item $X \mapsto \log \det \left(\frac{X+Y}{2}\right)$ for fixed $Y \in \pd$
        \item $X \mapsto \log \det \left(X^{r}Y \right)$ for fixed $Y \in \pd$ and $r \in \{-1, 1\}$
        \item $X \mapsto \log \det \left(\sum_{i=1}^n Y_i X^{r} Y_i^\top \right)$ for $\{Y_1, \ldots, Y_n\} \subseteq \pd$ and $r \in \{-1,1\}$.
    \end{enumerate}
    Moreover, the following map can be seen as a special case of (3). 
    \begin{enumerate}\setcounter{enumi}{3}
        \item Let $y_i \in \real^d \setminus \{0\}$ for $i = 1, \ldots, m$. The function 
        \[
        X \mapsto \log \left(\sum_{i=1}^m y_i^\top X y_i \right)
        \]
        is g-convex with respect to the canonical Riemannian metric.
    \end{enumerate}
    We provide an additional proof that this function is g-convex in Appendix~\ref{app:g_cvx_different_metrics}.
\end{example}
\begin{example}
The following maps are g-convex.
    \begin{enumerate}
        \item $g(X) = \sum_{i=1}^k \lambda_i^\downarrow(X^{-1})$ for $k = 1, \ldots, d.$
        \item $g(X) = \sum_{i=1}^k \log\left(\lambda_i^\downarrow(X^{-1})\right)$ for $k = 1, \ldots, d.$
        \item $g(X) = \log \det \left(\frac{X^{-1} + Y}{2}\right)$ for fixed $Y \in \pd$.
    \end{enumerate}
\end{example}
The following result generalizes Proposition~\ref{prop:logdet_gcvx} beyond the $\log \det (\cdot)$ function and also relaxes the strict positivity to positivity.
\begin{prop}[Theorem 15~\citep{sra2015conic}]\label{prop:sra_thm15}
Let $h: \pd \to \real$ be non-decreasing and g-convex. Let $r \in \{-1, 1\}$ and let $\Phi$ be a positive linear map. Then $\phi(X) = h\left(\Phi(X^r)\right)$ is g-convex with respect to the canonical Riemannian metric.
\end{prop}
\begin{example}[Examples of Proposition~\ref{prop:sra_thm15}]
Fix some $Y \in \pd$. Then the following results following directly from Proposition~\ref{prop:sra_thm15}.
\begin{enumerate}
    \item Let $h(X) = \tr(X^\alpha)$ for $\alpha \geq 1$ and $\Phi(X) = \sum_i Y_i^\top X Y_i$ then  $X \mapsto \tr\left( \sum_i Y_i^\top X^r Y_i\right)^\alpha$ is g-convex.
    \item Let $h(X) = \log \det (X)$ and $\Phi(X) = \sum_i Y_i^\top X Y_i$ then $X \to \log \det\left(\sum_i Y_i^\top X Y_i\right)$ is g-convex.
    \item Let $M \succeq 0$. Let $h(X) = \log \det(X)$ and $\Phi(X) = X \odot M$ then
    $X \mapsto \log \det \left( X \odot M\right)$
    is g-convex.
\end{enumerate}
\end{example}

We can extend the previous proposition to \textit{positive affine operators} which we now define.

\begin{definition}[ Positive Affine Operator]
    Let $B \succeq 0$ be a fixed symmetric positive semidefinite matrix and $\Phi: \pd \to \pd$ be a positive linear operator. Then the function $\phi:\pd \to \pd$ defined by 
    \[
    \phi(X) \defas \Phi(X) + B
    \]
    is an \textit{positive affine operator}.
\end{definition}

\begin{prop}[Geodesic Convexity of Positive Affine  Maps]\label{prop:gcvx_affine_positive}

    Let  $\phi(X) \defas \Phi(X) + B$ where $\Phi(X)$ is a positive linear map and $B \succeq 0$. 
        Let $f: \pd \to \mathbb{P}_m$ be g-convex and monotonically increasing, i.e., $f(X) \preceq f(Y)$ whenever $X \preceq Y$. Then the function
        $g(X) \defas f\left( \phi(X)\right)$
        is g-convex.
\end{prop}

\begin{example}
Let $B \succeq 0$ and $Y_i \in \pd$ for $i = 1, \ldots, n$ be fixed matrices. 
    \begin{enumerate}
    \item  $X \mapsto \tr\left(B +  \sum_i Y_i^\top X^r Y_i\right)^\alpha$ is g-convex.
    \item $X \mapsto \log \det\left( B + \sum_i Y_i^\top X Y_i\right)$ is g-convex.
    \item Let $M \succeq 0$. The map
    $X \mapsto \log \det \left(B +  X \odot M\right)$
    is g-convex.
\end{enumerate}

\end{example}

The following result provides a means for constructing geodesically convex  \textit{logarithmic tracial} functions.

\begin{theorem}[Theorem 17~\citep{sra2015conic}]\label{theorem:sra_logtrace}
    If $f: \real \to \real$ is Euclidean convex, then the function $\phi(X) = \sum_{i=1}^k f \left(\log \lambda^\downarrow_i(X)\right)$ is g-convex for each $1 \leq k \leq d$ where $\lambda_i^\downarrow(X)$ denotes the ordered spectrum of $X$, i.e., $\lambda_1^\downarrow(X) \geq \lambda_2^\downarrow(X) \cdots \geq \lambda_d^\downarrow(X)$. Moreover, if $h: \real \to \real$ is non-decreasing and Euclidean convex, then $\phi(X) = \sum_{i=1}^k h(|\log \lambda_i^\downarrow(X)|)$ is g-convex for each $1 \leq k \leq n$.
\end{theorem}

\subsubsection{Symmetric Positive Definite Manifold Atoms}\label{sec:atoms}
Geodesically convex functions in DGCP are constructed via compositions and transformations of basic geodesically convex functions, so-called \textit{atoms}. In this section, we provide a foundational set of geodesically convex functions defined on the manifold of symmetric positive definite matrices.


In DGCP, the atoms are either g-convex or g-concave in their argument. Moreover, each atom has a designated curvature, either \code{GIncreasing} or \code{GDecreasing}. This monotonicity property relies on a partial order relation on the symmetric positive definite matrices, induced by the \emph{Löwner order} (See Definition~\ref{def:loewner_order}).

This motivates the following definition:
\begin{definition}
    A function $f:\pd \to \pd$ 
    is \code{GIncreasing} if it satisfies 
    $f(A) \succeq f(B)$
    whenever $A \succeq B$. 
\end{definition}
In the following, we list our basic set of DGCP atoms. We defer all proofs of g-convexity to Appendix~\ref{app:gcvx_atoms}. We emphasize that our framework has a \emph{modular} design, which allows for implementing additional atoms as needed. 

\subsubsection{Scalar-valued atoms}
We begin with a set of \textit{scalar-valued} DGCP atoms.
\paragraph{Log Determinant.}
 \texttt{LinearAlgebra.logdet(X)} represents the log-determinant function $\log \det: \pd \to \real_{++}$. This is an example of an atom that is \code{GLinear} (i.e. both g-convex and g-concave) and \code{GIncreasing}. It is concave in the Euclidean setting. 
 
\paragraph{Trace.} \code{LinearAlgebra.tr(X)} sums the diagonal entries of a matrix. It has \code{GConvex} curvature and is \code{GIncreasing}. It is affine in the Euclidean setting.

\paragraph{Sum of Entries.} 
\code{sum(X)} will sum the entries of X, i.e., returns $\sum_{i,j=1}^d X_{ij}$. It has \code{GConvex} curvature and is \code{GIncreasing}. It is affine in the Euclidean setting.  

\paragraph{S-Divergence.}
\code{sdivergence(X,Y)} is defined as 
\begin{equation}\label{eq:sdiv}
    \code{sdivergence(X,Y)} := \log \det \left( \frac{X+Y}{2} \right) - \frac{1}{2}\log \det (XY).
\end{equation}
This function is jointly geodesically convex, i.e., it is has \code{GConvex} curvature in both $X$ and $Y$ and is \code{GIncreasing}. It is non-convex in the Euclidean setting.

\paragraph{Riemannian Metric.}
\code{Manifolds.distance(X,Y)} returns the distance with respect to the \textit{affine-invariant} metric. 
\[
\code{Manifolds.distance(X,Y)} := \left\|\log \left(Y^{-1/2}X Y^{-1/2}\right)\right\|_F.
\] 
It is \code{GConvex} and is neither \code{GIncreasing} nor \code{GDecreasing} hence its monotonocity is unknown i.e. \code{GAnyMono}.

\paragraph{Quadratic Form.}
Fix $h \in \real^d$. The following function is g-convex $\code{quad\_form(h, X)} = h^\top X h$ and \code{GIncreasing}. It is also convex in the Euclidean setting.

\paragraph{Spectral Radius.} We define
\[\code{LinearAlgebra.eigmax(X)} := \sup_{\|y\|_2 = 1}y^\top X y  \; ,\]
as the function that takes in $X \in \pd$ and returns the maximum eigenvalue of $X$. This is a g-convex function and \code{GIncreasing}. It is also convex in the Euclidean setting.

\paragraph{Log Quadratic Form}

 Let $h_i \in \real^d$ be nonzero vectors for $i = 1, \ldots, n$. Then  
\[
\code{log\_quad\_form(\{h\_1 \ldots, h\_n\}, X)} = \log \left(\sum_{i=1}^n h_i^\top X^{r} h_i \right) \; , \qquad r \in \{-1, 1\}.
\]
This is a g-convex function and \code{GIncreasing}. See Lemma 1.20 in \citep{wieselstructuredcovariance}. It is non-convex in the Euclidean setting.

\begin{definition}[Symmetric Gauge Functions] 
     A map $\Phi:\real^d \to \real_+$ is called a symmetric gauge function if 
    \begin{enumerate}
        \item $\Phi$ is a norm;
        \item $\Phi(Px) = \Phi(x)$ for all $x \in \real^n$ and all $n\times n$ permutation matrices $P$. This is known as the \textit{symmetric} property;
        \item $\Phi(\alpha_1 x_1, \ldots, \alpha_n x_n) = \Phi(x_1, \ldots, x_n)$ for all $x \in \real^n$ and $\alpha_k \in \{\pm 1\}$. This is known as the \textit{gauge invariant} or \textit{absolute} property.
    \end{enumerate}
\end{definition}

\begin{prop}[Symmetric Gauge Functions are g-convex~\citep{struct-reg}]\label{prop:sgf_gvx}
 Let $\Phi: \real^d \to \real$ be a symmetric gauge function. Then the function  $f(A) := \Phi(\lambda(A))$ is geodesically convex where $\lambda(A) = \{\lambda_1(A), \ldots, \lambda_d(A)\} \in \real^d$ is the eigenspectrum of $A$.
\end{prop}

\begin{remark}
    For a symmetric gauge function $\Phi: \real^d \to \real$ and a matrix $A \in \pd$ we use the notation $\Phi(A)$ to mean $\Phi(\lambda(A))$, i.e. $\Phi(A)$ acts on the eigenspectrum of $A$.
\end{remark}

\begin{example}[Symmetric Gauge Functions]
    The two canonical symmetric gauge functions are the Ky Fan and $p$-Schatten norm.
    \begin{enumerate}
        \item The $k$-\emph{Ky Fan function} of $X$ is the sum of the top $k$ eigenvalues, i.e., 
        \[
        \Phi(X) = \sum_{i=1}^k \lambda_i^\downarrow(X) \; , \qquad 1 \leq k \leq d \; ,
        \]
        where $\lambda_i^\downarrow(X)$ is the sorted spectrum of $X$. The atom for $k$-\emph{Ky Fan function} in our library is available as \code{eigsummax(X, k)}. 
        \item The \emph{$p$-Schatten norm} for $p \geq 1$ is defined as 
        \[
        \Phi(X) = \left(\sum_{i=1}^d \lambda^p_i(X)\right)^{\frac{1}{p}} \; .
        \]

        The corresponding atom in our library is provided as \code{schatten\_norm(X, p)}.
    \end{enumerate}
\end{example}

\begin{example}
    The following logarithmic symmetric gauge functions are g-convex by applying Theorem~\ref{theorem:sra_logtrace}. They can be used with the \code{sum\_log\_eigmax} atom in our implementation.
    \begin{enumerate}
        \item Let $f(t) = t$ be the identity function in Theorem~\ref{theorem:sra_logtrace}. Then 
        \[
        \phi(X) = \sum_{i=1}^k \log \lambda_i^\downarrow(X) = \Phi(\log(X)) \; , \qquad 1 \leq k \leq d \; ,
        \]
    is g-convex where $\Phi(\cdot)$ is the \textit{k}-Ky fan norm. 
    \item Let $f(t) = t^p$ for $p \geq 1$ in Theorem~\ref{theorem:sra_logtrace}. Then the function
    \[
    \phi(X) = \sum_{i=1}^k \left(\log \lambda^\downarrow_i(X)\right)^p \; , \qquad 1 \leq k \leq d \; ,
    \]
    is g-convex. 
    \end{enumerate}

\end{example}

\paragraph{Positive Affine Maps}

The results in \ref{prop:gcvx_affine_positive} can be leveraged using the \code{affine\_map} atom in our accompanying package.

\subsubsection{Matrix-valued atoms}
Our framework further incorporates a set of \emph{matrix-valued DGCP atoms}, which are crucial for verifying the g-convexity of matrix-valued objectives and constraints.

\paragraph{Conjugation.} Let $X \in \pd$ and $A \in \R^{n \times n}$ then $\code{conjugation}(X, A) = A^\top X A$.
This atom has \code{GConvex} curvature and is \code{GIncreasing}. It is Affine in the Euclidean setting.

\paragraph{Adjoint.} Let $X \in \pd$ then $\code{adjoint(X)} = X^\top$ has \code{GConvex} curvature and \code{GIncreasing}. It is Affine in the Euclidean setting.

\paragraph{Inverse.} Let $X \in \pd$ then $\code{inv(X)} = X^{-1}$ has \code{GConvex} curvature and \code{GDecreasing}. It is also Convex in the Euclidean setting.

\paragraph{Hadamard product.} Let $X \in \pd$ then $\code{hadamard\_product(X, B)} = X \odot B$ has \code{GConvex} curvature and \code{GIncreasing}. It is affine in the Euclidean setting.

\subsection{Lorentz Model}\label{sec:Lorentz_model}

To illustrate the versatility of the DGCP framework, we provide DGCP rules and atoms for the Lorentz model as discussed in Section~\ref{sec:Riemannian_Geometry}. 

We mainly focus on geodesic convexity results of quadratic functions. The homogeneous quadratic function of the form $f(p) = p^\top A p$ was recently studied \citep{Ferreira2022}. Unlike Euclidean space, the geodesic convexity of homogeneous and nonhomogeneous quadratic functions are non-trivially different. Results of geodesic convexity for the nonhomogeneous case $f(p) = p^\top A p + b^\top p + c$ was also recently established \citep{Ferreira2023_nonhomogeneous}.

\paragraph{Notation.} For a symmetric matrix $A \in \real^{(d+1) \times (d+1)}$ and vector $b \in \real^{d+1}$ we will make use of the decomposition 

\[
\begin{gathered}
     A:=\left(\begin{array}{cc}
\bar{A} & \bar{a} \\
\bar{a}^{\top} & \sigma
\end{array}\right), \quad \bar{A} \in \mathbb{R}^{d \times d}, \quad \bar{a} \in \mathbb{R}^{d \times 1}, \quad \sigma \in \mathbb{R} \\
\text{and} \qquad b:=\binom{\bar{b}}{b_{n+1}} \in \mathbb{R}^{d+1}, \quad \bar{b} \in \mathbb{R}^d, \quad b_{d+1} \in \mathbb{R}.
\end{gathered}
\]

\subsubsection{Lorentzian Rules}

The following rule allows us to construct geodesically convex nonhomogenous functions from geodesically convex homogenous functions. 

A square matrix $A$ is called $\partial \mathcal{L}$-copositive if $p^\top A p \geq 0$ for all $p \in \partial \mathcal{L}.$

\begin{prop}[Proposition 3.5~\cite{Ferreira2023_nonhomogeneous}]\label{prop:nonhom_hom}
    Let $A=A^{\top} \in \mathbb{R}^{(n+1) \times(n+1)}, b \in \mathbb{R}^{n+1}, c \in \mathbb{R}, f: \mathbb{H}^n \rightarrow \mathbb{R}$ be defined by $f(p)=p^{\top} A p+b^{\top} p+c$ and $h: \mathbb{H}^n \rightarrow \mathbb{R}$ be defined by $h(p)=p^{\top} A p$. The following are equivalent
    \begin{enumerate}
        \item \text The function $f$ is geodesically convex.
        \item The function $h$ is geodesically convex with $b \in \mathscr{L}$ where $\mathscr{L} := \{x \in \real^{d+1}: x^\top J x \leq 0, x_{d+1} \geq 0\}$ is known as the \emph{Lorentz cone.}
        \item  $A$ is $\partial \mathscr{L}$-copositive and $b \in \mathscr{L}$.
    \end{enumerate}
\end{prop}

The previous proposition states that if we know the homogeneous quadratic function $h(p) = p^\top A p$ is geodesically convex and $b$ lies in the Lorentz cone then the corresponding nonhomogeneous function $f(p) = p^\top A p + b^\top p + c$ is geodesically convex.

\begin{example}\label{ex:lorentz_cone_set}
    Observe that the set $C := \{b \in \real^{d+1} : \|\bar{b}\|_2 \leq b_{d+1}, \ b_{d+1} \geq 0\} \subseteq \mathscr{L}$. 
    
    Let $A = A^\top \in \real^{(d+1) \times (d+1)}$.  If $h:\mathbb{H}^d \to \real$ defined by $h(p) = p^\top A p$ is geodesically convex then $f(p) = p^\top A p + b^\top p + c$ is geodesically convex for all $b \in C$.
\end{example}

\begin{prop}[Theorem 3.1~\cite{Ferreira2023_nonhomogeneous}]
Let $A=A^{\top} \in \mathbb{R}^{(d+1) \times(d+1)}$ be a nonzero matrix, $b \in \mathbb{R}^{n+1}, c \in \mathbb{R}$, $f: \mathbb{H}^d \rightarrow \mathbb{R}$ be defined by $f(p)=p^{\top} A p+b^{\top} p+c$ and $g: \mathbb{H}^d \rightarrow \mathbb{R}$ be defined by $g(p)=p^T p+b^{\top} p+c$. If $f$ is geodesically convex then the function $h: \mathbb{H}^d \rightarrow \mathbb{R}$ defined by

$$
h(p)=p^T p+\left(b^A\right)^{\top} p+c
$$

is geodesically convex, where

$$
b^A=\frac{1}{\|A\|_2} b .
$$

\end{prop}

Next, we note that compositions with the Lorentz group preserves geodesic convexity.

\begin{definition}[Lorentz Group]
Let $J := \diag(1, \ldots, 1, -1) \in \real^{(d+1) \times (d+1)}$. The Lorentz group $G_\Lorentz$ is defined as 
\[
G_{\Lorentz}:=\left\{Q \in \mathbb{R}^{(d+1) \times(d+1)}: Q^{\top} \mathrm{J} Q=\mathrm{J}\right\} .
\]
\end{definition}

The following subgroup of the Lorentz group contains global isometries of $\mathbb{H}^d$.

\begin{definition}[Orthochronous Lorentz Group]
The orthochronous Lorentz group denoted by $\mathcal{O}^+(1,d)$ is a subgroup of the Lorentz group that preserves the positivity of the last coordinate. That is 
\[
\mathcal{O}^+(1,d) := \{Q \in G_\Lorentz: (Qx)_{d+1} > 0 \text{ for all } x \in \real^{d+1} \text{ with } x_{d+1} > 0\}.
\]
\end{definition}

\begin{example}[Lorentz Group Elements]
We provide examples of Lorentz group elements. 
\begin{enumerate}
    \item \textbf{Identity.} $I \in \mathcal{O}^{+}(1,d)$ and $-I \in G_\Lorentz.$
   
    \item \textbf{Spatial Inversion.}  $O = \diag(-1,\ldots, -1, 1) \in \real^{(d+1)\times(d+1)} \in \mathcal{O}^+(1,d)$
    \item \textbf{Time Reversal.} $Q = \diag(1, \ldots, 1, -1) \in \real^{(d+1)\times(d+1)} \in G_\Lorentz$
    \item \textbf{Lorentz Boost.} 
    \[
\mathcal{O}_{\text{boost}} =
\begin{pmatrix}
I_{d-1} & 0 & 0 \\
0 & \cosh(\phi) & -\sinh(\phi) \\
0 & -\sinh(\phi) & \cosh(\phi)
\end{pmatrix} \in \mathcal{O}^+(1,d).
\]
    

     \item Let $x \in \real^{d+1}$ such that $\|x\|_\Lorentz > 0$. 
    \[
    Q := I - \left(\frac{2}{\|x\|_\Lorentz} \right)^2 x x^\top J \in G_\Lorentz
    \]
    \item Let $x,y \in \real^{d+1}$ such that $\|x\|_\Lorentz = \|y\|_\Lorentz = 1$. Then 
    \[
    Q=\mathrm{I}+2 y x^{\top} J-\left( \frac{1} {1+x^{\top} J y}\right)(x+y)(x+y)^{\top} J \in G_{\mathcal{L}} .
    \]
\end{enumerate}
    
\end{example}

\begin{prop}[\citep{Ferreira2022}]\label{prop:lorentz_composition}
    Let $\mathcal{C} \subseteq \mathbb{H}^d$ be a hyperbolically convex set, $Q \in G_{\mathcal{L}}$ and $\mathcal{D}:=$ $\left\{Q^{-1} p: p \in \mathcal{C}\right\}$. The function $f: \mathcal{C} \rightarrow \mathbb{R}$ is geodesically convex if and only if $f \circ Q: \mathcal{D} \rightarrow \mathbb{R}$ defined by $f \circ Q(q):=f(Q q)$ is geodesically convex.
\end{prop}

\begin{remark}
    Let $O \in \mathcal{O}^{+}(1,d)$ be an element of the orthochronous Lorentz group. If $f:\mathbb{H}^d \to \real$ is g-convex then $g(q) \defas f(O q): \mathbb{H}^d \to \real$ is g-convex.
\end{remark}

\subsubsection{Lorentzian Atoms}\label{sec:lorentzian_atoms}

\textbf{Lorentzian Distance.} Let $q \in \mathbb{H}_d$. The function $d_\mathcal{L}(\cdot, q): \mathbb{H}_d \to \real$ defined by
\[d_\mathcal{L}(p,q) := \operatorname{arcosh}(-\langle p, q \rangle_\mathcal{L})\]
is geodesically convex.

 \textbf{Log-Barrier~\citep{Ferreira2022}.} Let $a = (0, \ldots, 0, 1) \in \real^{d+1}$ and define the geodesically convex set 
\[
\mathcal{C} := \{p \in \mathbb{H}^d: p_1 > 0, \ldots, p_n >0 \}. 
\]
The log-barrier function defined as $\psi: \mathcal{C} \to \real$ defined by 
\[
\psi(p)=-\log  (-1-\langle a, p\rangle_\Lorentz)
\]
is geodesically convex. 

\textbf{Homogeneous Positive Semidefinite \citep{Ferreira2022}.} Let $A \in \mathbb{R}^{(d+1)\times(d+1)}$ be a positive semidefinite matrix. Then the function $f:\mathbb{H}_d \to \real$ defined by $f(p) = p^\top A p$ is geodesically convex.

\textbf{Homogeneous Diagonal ~\citep{Ferreira2022}}.
    Take $A = \diag(a_1, \ldots, a_d, a_{d+1})$ and assume $a_{\min} + a_{d+1} \geq 0$ where $a_{\min} = \min\{a_1, \ldots, a_n\}$. Then 
    \[
    f(p) = \sum_{i=1}^n a_i p_i^2
    \]
    is g-convex.

\paragraph{Least Squares Problem.}
Suppose $X \in \mathbb{R}^{n \times (d+1)}$ and $y \in \real^{n}$. We define the least squares problem on $\mathbb{H}_d$ to be 
\[
\min_{p \in \mathbb{H}_d}f(p) = \|y - Xp\|_2^2 = y^\top y - 2y^\top X p + p^\top X^\top X p.
\]
Applying Proposition~\ref{prop:nonhom_hom} we can conclude $f:\mathbb{H}_d \to \real$ is geodesically convex if $A = X^\top X$ is $\partial \Lorentz$-copositive and $b=-2X^\top y \in \Lorentz$. Since $X^\top X$ is positive semidefinite, copositivity trivially follows. The constraint on the linear term $b$ lying in the cone $\Lorentz$ can be equivalently expressed as 
\begin{equation}\label{eq:hyperbolic_least_squares}
    \sum_{i=1}^d \left(X^\top y\right)_i^2 \leq \left(X^\top y\right)_{d+1}^2 \qquad \text{and} \qquad \left(X^\top y\right)_{d+1} \leq 0.
\end{equation}

\eqref{eq:hyperbolic_least_squares} places a non-trivial constraint on $X$ and $y$. Namely, the first inequality implies that the project of $y$ onto the first $d$ columns of $X$ must not exceed the absolute magnitude of its $y$ projected onto the last column of $y.$ This can be satisfied if the first $d$ columns are sufficiently sparse or is nearly orthogonal to $y$. The second inequality says the dot product between $y$ and the last column of $X$ must be non-positive.

Often, one includes a bias term in linear regression which results in the last column of $X$ to be the vector of 1's. Then \eqref{eq:hyperbolic_least_squares} becomes 
\begin{equation}
     \left \|X_{:, 1:d}^\top y \right \|_2 \leq \left | \sum_{i=1}^n y_i \right | \qquad \text{and} \qquad \sum_{i=1}^n y_i \leq 0.
\end{equation}
where $X_{:, 1:d} \in \mathbb{R}^{n \times d}$ denotes the matrix constructed from the first $d$ columns of $X$.


\section{Implementation}

The implementation of disciplined geodesically convex programming (DGCP) in this work is based on the foundation of symbolic computation and rewriting capability of the \textsl{Symbolics.jl} package~\citep{gowda2021high}. 

Each expression written with \textsl{Symbolics} is represented as a tree, where the nodes represent functions (or atoms), and the leaves represent variables or constants (see example in Figure~\ref{fig:exptree}).  This representation enables the propagation of function properties, such as curvature and monotonicity, through the expression tree. 
\begin{figure}[h!]
    \centering
    \begin{forest}
        for tree={
            grow=south,
            parent anchor=south,
            child anchor=north,
            edge path={
                \noexpand\path [draw, \forestoption{edge}]
                (!u.parent anchor) -- +(0,-5pt) -| (.child anchor)\forestoption{edge label};
            },
            l sep=1.5cm,
            s sep=2cm,
            anchor=center,
            align=center,
            edge={-latex},
            inner sep=2pt,
            text centered,
            draw,
            rounded corners,
            font=\footnotesize
        }
        [{$ADD$\\{\color{Plum}\footnotesize GConvex, AnySign}}
            [{$MUL$\\{\color{Plum}\footnotesize GLinear, Negative}}
                [{$-1$}]
                [{$\text{logdet}$\\{\color{Plum}\footnotesize GLinear, Positive}}
                    [{$X$}]
                ]
            ]
            [{$\text{logdet}$\\{\color{Plum}\footnotesize GConvex, Positive}}
                [{$\text{conjugation}$\\{\color{Plum}\footnotesize GConvex, Positive}}
                    [{$X$}]
                    [{$A_{5\times5}$}]
                ]
            ]
        ]
    \end{forest}
    \caption{Expression tree for the problem of computing Brascamp-Lieb constants given in Eq.~\ref{eqn:brascamplieb}. The properties of the components are propagated up through the tree using the known properties of the atoms that make up the expression, giving the final geodesic curvature as \code{GConvex} and sign of the function as \code{AnySign}.}
    \label{fig:exptree}
\end{figure}

Previous implementations of disciplined programming, in CVXPY~\citep{diamond2016cvxpy} and \textsl{Convex.jl}~\citep{udell2014convex}, define a class in the Object-Oriented Programming 
sense for each atom. We take a different approach in our DGCP implementation. The relevant properties, such as domain, sign, curvature and monotonicity, are added as metadata to the leaves, and then propagated by looking up the corresponding property for every atomic function. The DGCP compliant rules are implemented using the rule-based term rewriting provided by \textsl{SymbolicUtils.jl}~\citep{symutils}. For analyzing arbitrary expressions, the properties are recursively added on by a postorder tree traversal. This approach allows for greater flexibility and modularity in defining new atoms and rules, enabling the incorporation of domain-specific atoms. Since the atoms are directly the Julia functions, the DGCP implementation avoids the need to create and maintain implementations of numerical routines.

\subsection{Atom Library}

The atoms in DGCP are stored as a key-value pair in a dictionary. Wherein the key is the Julia method corresponding to the atom and the value is a tuple containing the manifold, the sign of the function, and its known geodesic curvature and the monotonicity. For a Julia function to be compliant with the rule propagation discussed in the next sub-section, it needs to be a registered primitive in \textsl{Symbolics} through the \verb|@register_symbolic| macro from \textsl{Symbolics}.         
For example, the \texttt{logdet} atom representing the log-determinant of a symmetric positive definite matrix, implemented with the function from the \textsl{LinearAlgebra} standard library of Julia, is defined as follows:

\begin{figure}
    \centering
    \includegraphics[width=\linewidth]{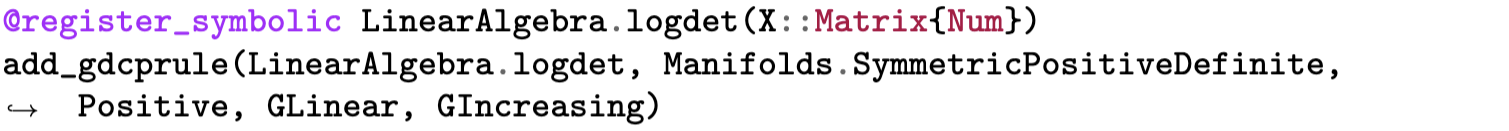}
    \caption{The \texttt{logdet} atom is defined on the \texttt{Manifolds.SymmetricPositiveDefinite} manifold, has a positive sign, is geodesically linear, and is geodesically increasing.}
    \label{logdetatom}
\end{figure}

Some atoms in DGCP do not have preexisting implementations in Julia, so first a function is defined for it and the same machinery as before is then used to register. For instance, the \texttt{conjugation} atom is defined as follows:




\begin{figure}
    \centering
    \includegraphics[width=\linewidth]{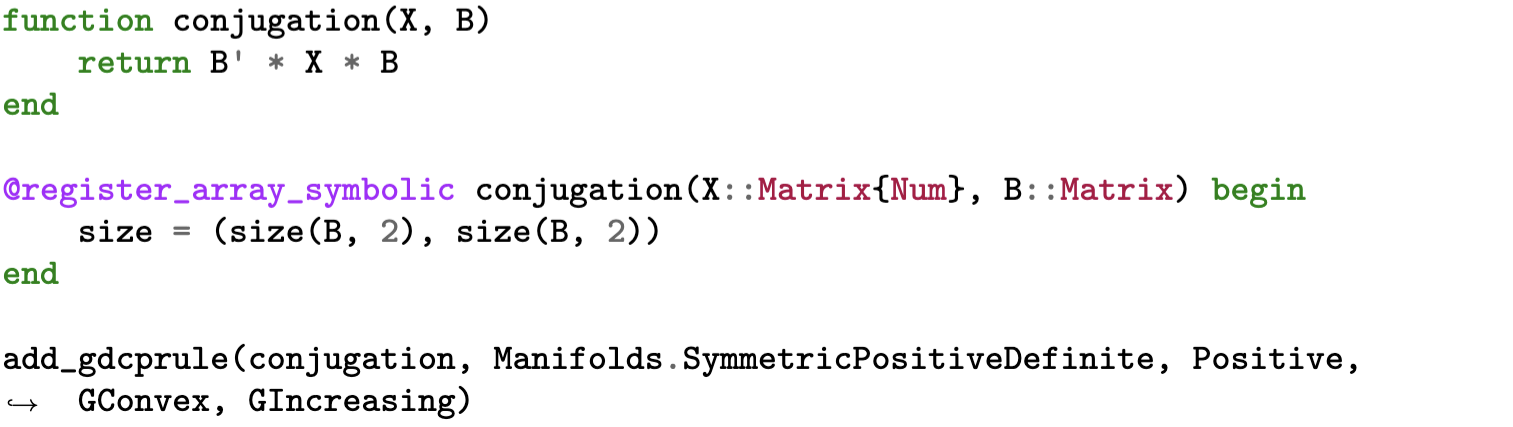}
    \caption{The \texttt{conjugation} atom is defined on the \texttt{Manifolds.SymmetricPositiveDefinite} manifold, has a positive sign, is geodesically convex, and is geodesically increasing.}
    \label{conjugation}
\end{figure}

The extensibility of the atom library is an important feature of this implementation. Users can define atoms and specify their properties using the provided macros and functions, allowing the incorporation of domain-specific atoms and the ability to handle a wide range of optimization problems. The modular design of the atom library enables the addition of new atoms without modifying the core implementation and allows more disciplined programming paradigms to be implemented similarly.

\subsection{Rewriting System for Rule Propagation}

The DGCP compliant ruleset \ref{sec:rules} lends itself naturally to a rewriting system \citep{dershowitz1990rewrite}, as has been shown before for DCP \citep{agrawal2018rewriting}. The \textsl{SymbolicUtils.jl} package provides the rewriting infrastructure that enables the application of DGCP rules to symbolic expressions.

In the DGCP implementation, rewriting is employed to propagate the mathematical properties of functions as metadata. The rewriting system applies the rules using a post-order traversal of the expression tree, ensuring that the properties of subexpressions are propagated before determining the properties of parent expressions.

The DGCP ruleset is implemented using the \texttt{@rule} macro. For example, the following rule propagates the curvature through addition of subexpressions:


\begin{figure}
    \centering
    \includegraphics[width=\linewidth]{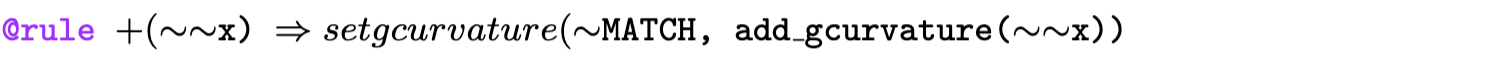}
    \caption{Using the \texttt{@rule} macro for propagating Geodesic Curvature through addition.}
    \label{rulecurv}
\end{figure}

This rule matches an addition expression \texttt{+($\sim\sim$x)} and sets the curvature of the matched expression (\texttt{$\sim$MATCH}) to the result of the \texttt{add\_gcurvature} function applied to the subexpressions (\texttt{$\sim\sim$x}).

The rewriting and metadata propagation from \textsl{SymbolicUtils} allows for a declarative specification of the rules, reducing the lines of code required to implement the DGCP ruleset.

\subsection{Integration with Optimization Frameworks}

To leverage the DGCP in applications, we require an integration of our framework with manifold optimization software for solving the verified programs. This has been done with \textsl{OptimizationManopt}, which is the interface to \textsl{Manopt.jl} with the \textsl{Optimization.jl}~\citep{vaibhav_kumar_dixit_2023_7738525} package. This integration allows us to define the optimization problem, either with an algebraic or a functional interface, and perform this analysis to determine whether the objective function and/or constraints are geodesically convex.

During the initialization phase in \textsl{Optimization.jl}, the symbolic expressions for the objective function and constraints are generated by tracing through the imperative code with symbolic variables. This automatic generation of symbolic expressions allows for a transition from the optimization problem specification to the symbolic representation required for verification with DGCP. As mentioned above, this can also be done by using the algebraic interface, in which case the analysis still proceeds as before, except that symbolic tracing is not needed as the user already provides the expression.

The generated symbolic expressions are then leveraged to propagate the sign information and geodesic curvature using the \texttt{propagate\_sign}, and \texttt{propagate\_gcurvature} functions, and the user is informed if the problem can be recognized to be disciplined geodesically convex or otherwise (see~\ref{listing:verificationproblem}).
%



%

\begin{figure}
    \centering
    \includegraphics[width=\linewidth]{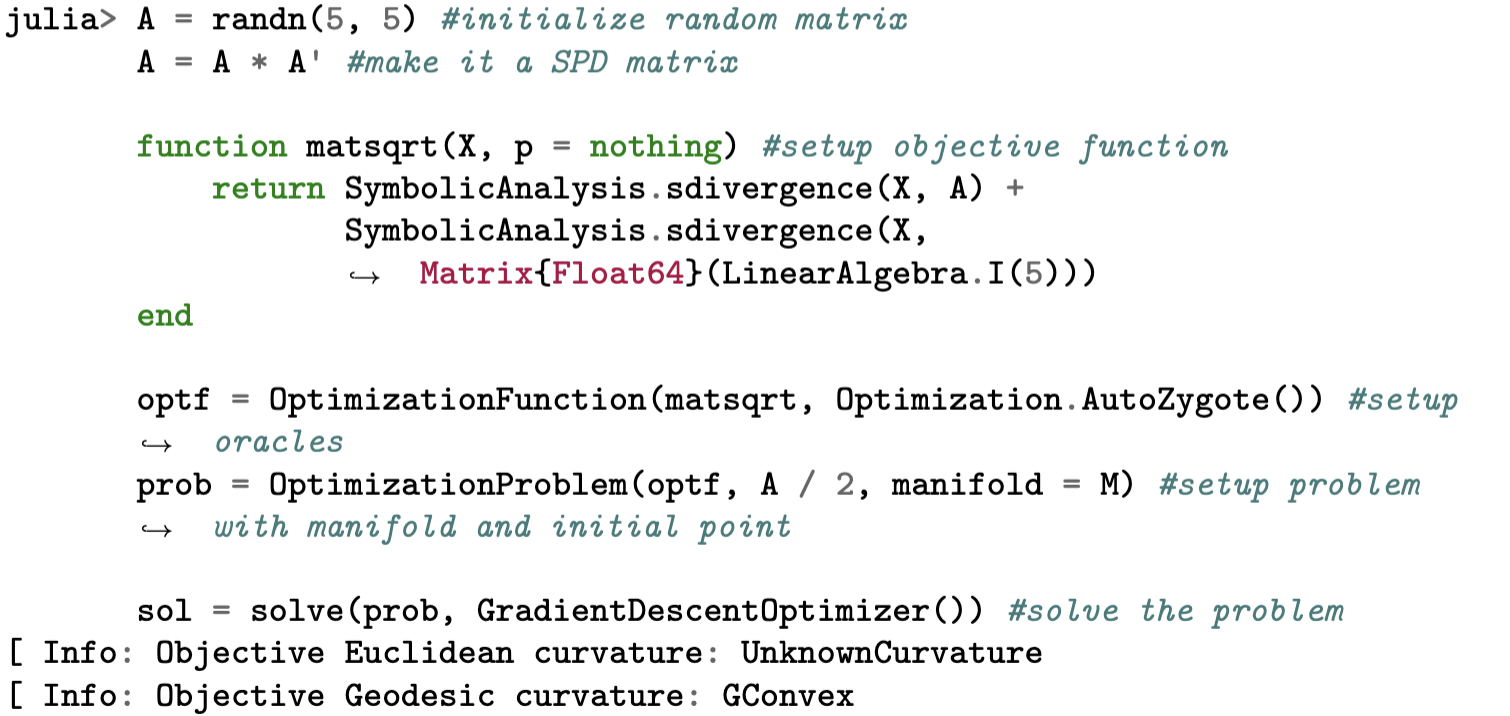}
    \caption{Solving the matrix square root problem in geodesically convex formulation from \citep{sra2015matrix} with Geodesic Convexity certificate.}
    \label{listing:verificationproblem}
\end{figure}

The program can be solved using a selected solver from \textsl{Manopt.jl}. The curvature propagation step described above gives us a certificate of Geodesic Convexity. Hence, in conjunction with \textsl{Manopt.jl}, DGCP provides a generic non-linear programming interface for Riemannian optimization with certificates of global optimality. 

\subsection{Performance Analysis}

To demonstrate the practical efficiency of our DGCP framework, we present an 
analysis of the runtime of the verification procedure for three representative g-convex problems of varying symbolic complexity. We measure the time required for DGCP to perform symbolic analysis and verify g-convexity, not the subsequent numerical optimization.
Our experiments were conducted on a MacBook Pro with an Apple M3 Pro processor (11 cores) using Julia v1.11.3 with the \textsl{SymbolicAnalysis.jl} package. Each measurement represents the median of 10 independent runs after a 
warm-up phase to eliminate compilation artifacts. 

\textbf{Tyler's M-Estimator} This example represents the most symbolically complex case, involving inverse matrix operations, logarithmic quadratic forms, and iterative summations over data points. The expression structure requires extensive symbolic analysis to verify the composition of multiple g-convex atoms through DGCP-compliant rules.

\textbf{Karcher Mean} This problem exhibits medium symbolic complexity, involving Riemannian distance computations and power operations. While simpler than Tyler's estimator, the expression still requires non-trivial symbolic analysis to verify g-convexity via distance-based atoms.

\textbf{Log-Determinant} This example serves as a baseline for simple expressions, consisting of a single atomic operation. Hence, the DGCP verification involves minimal symbolic analysis.

\begin{figure}[htbp]
\centering
\begin{subfigure}[b]{0.32\textwidth}
    \includegraphics[width=\textwidth]{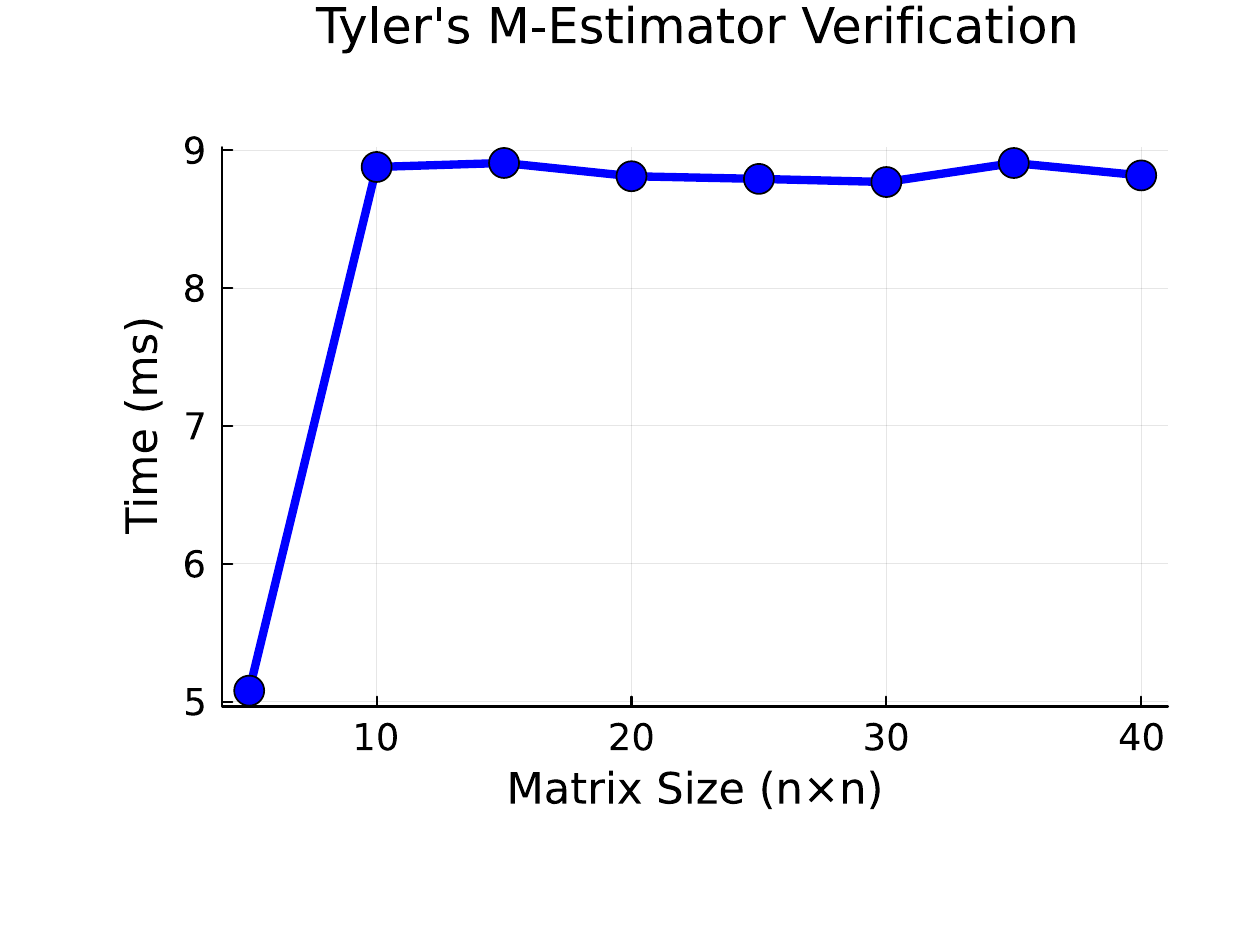}
    \caption{Tyler's M-Estimator}
    \label{fig:dgcp_tyler}
\end{subfigure}
\hfill
\begin{subfigure}[b]{0.32\textwidth}
    \includegraphics[width=\textwidth]{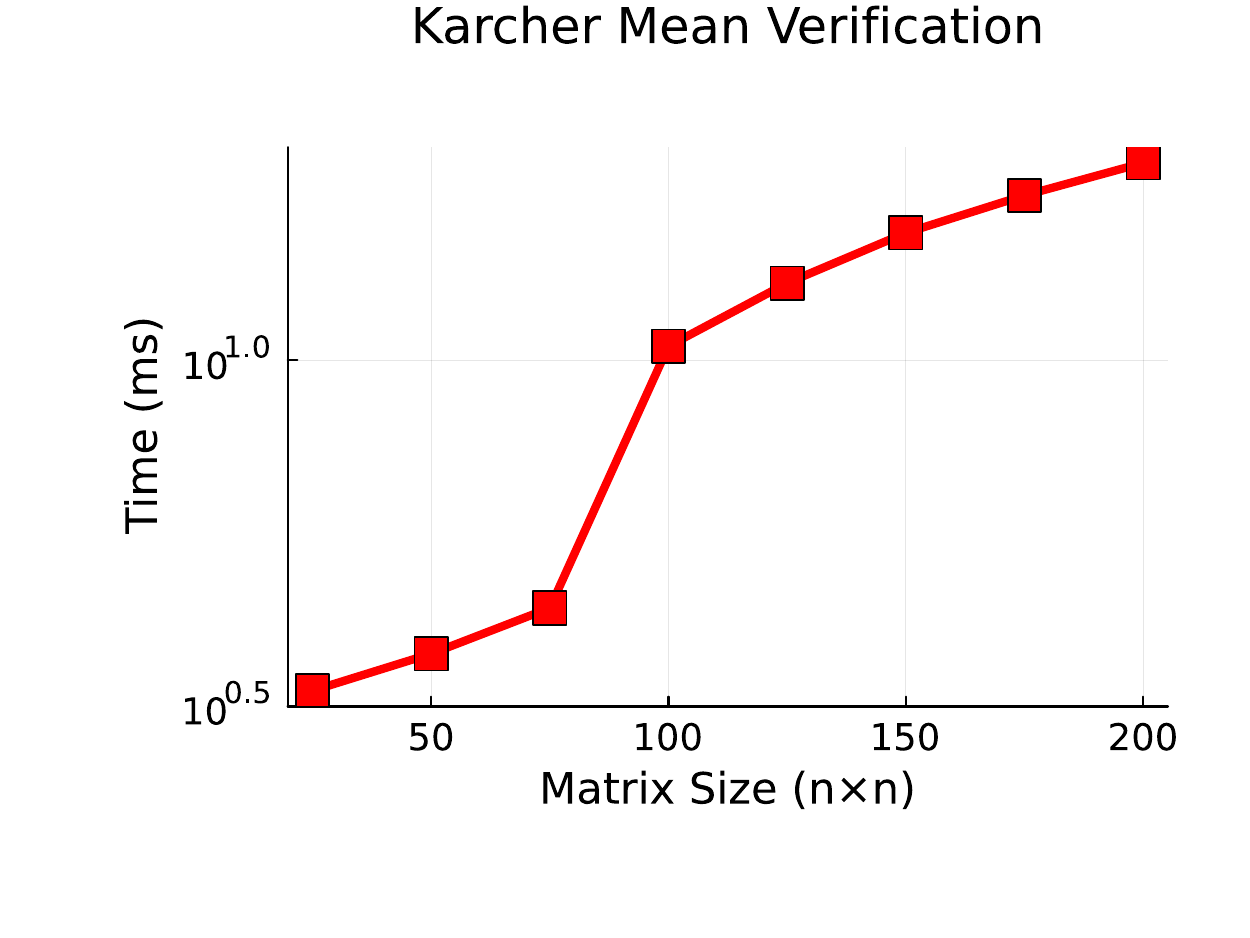}
    \caption{Karcher Mean (log scale)}
    \label{fig:dgcp_karcher}
\end{subfigure}
\hfill
\begin{subfigure}[b]{0.32\textwidth}
    \includegraphics[width=\textwidth]{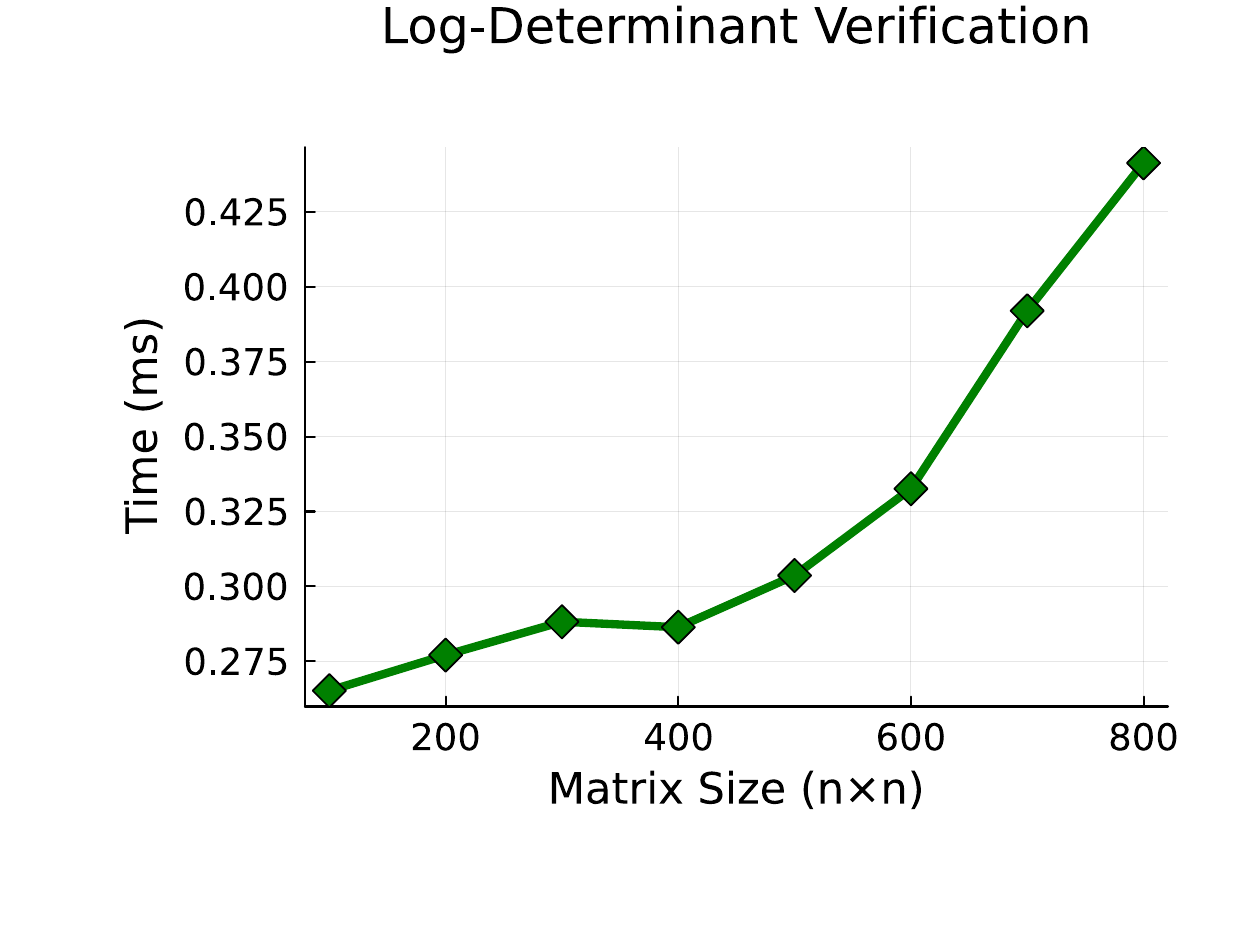}
    \caption{Log-Determinant}
    \label{fig:dgcp_logdet}
\end{subfigure}
\caption{\textbf{DGCP verification performance across symbolic complexity levels.} Times represent symbolic analysis duration, not numerical computation. Tyler's M-estimator requires the most complex symbolic verification ($\sim$8ms), involving matrix inversions and logarithmic operations. Karcher mean shows medium complexity ($\sim$0.5-5ms), while log-determinant verification completes in under 0.5ms as a single atomic operation. Verification time depends primarily on expression complexity rather than matrix dimensions.}
\label{fig:dgcp_performance}
\end{figure}

Our results demonstrate several key properties of DGCP. First, \textbf{symbolic complexity dominates matrix size} in determining verification time. Tyler's M-estimator consistently requires $\sim$8ms regardless of matrix dimensions from $5 \times 5$ to $40 \times 40$, while log-determinant verification remains under 0.5ms even for matrices up to $800 \times 800$. The slight variations observed within each problem type primarily reflect differences in symbolic expression structure (e.g., varying numbers of data points in Tyler's estimator) rather than numerical scaling effects. This behavior reflects the fact that DGCP analyzes symbolic expression trees rather than performing numerical matrix operations.

Second, \textbf{verification scales with expression complexity}, not problem size. 
Ordered by verification time, we see that 
Tyler's M-estimator requires more time than the Karcher mean, which requires more time than verifying the log-determinant. This directly corresponds to the number of symbolic operations and composition rules required for each verification problem: While Tyler's estimator involves 15 distinct symbolic operations (inverse, logarithm, quadratic forms, summations), the log-determinant requires only a single atomic operation lookup.

Third, \textbf{all verification times remain practically feasible}, completing in under 10ms even for the most complex expressions. This demonstrates that DGCP adds minimal computational overhead to the optimization workflow, making real-time g-convexity certification viable for applications in practice.



\subsection{Limitations of DGCP}
While DGCP successfully verifies g-convexity for a broad class of functions, 
the output 	``not g-convex'' may either indicates genuine non-g-convexity or that the program cannot be verified with existing atoms and rules. The latter case could be mitigated by adding further atoms and rules to expand the framework's scope. These characteristics resemble those of other disciplined programming frameworks.

Below, we illustrate these observations through examples, showing that DGCP exhibits the expected characteristics.

\paragraph{Products of Geodesically Convex Functions}
As proven in Apx.~\ref{app:g_cvx_different_metrics}, products do not preserve g-convexity. DGCP correctly identifies this:


\begin{figure}[!h]
    \centering
    \includegraphics[width=\linewidth]{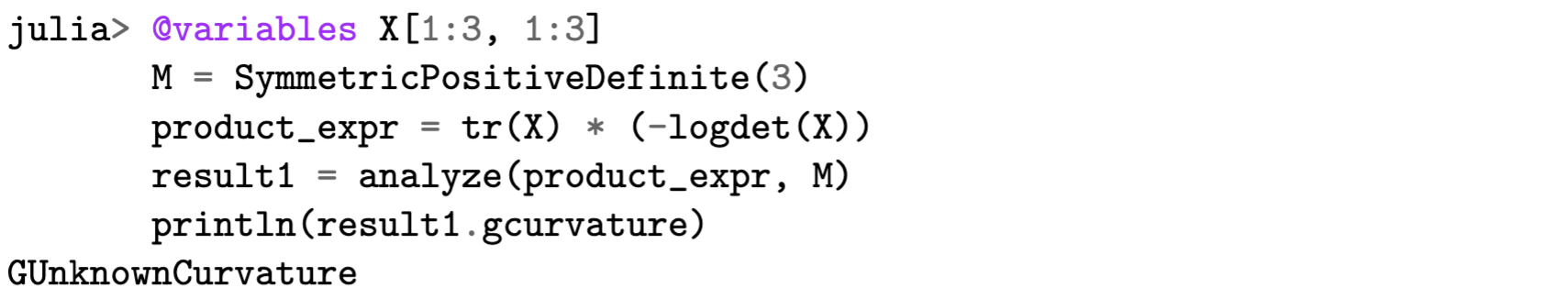}
\end{figure}

This demonstrates that DGCP's composition rules correctly capture that products do not preserve g-convexity.

\paragraph{Element-wise Matrix Norms.}
The element-wise 1-norm provides another instructive example. As shown in Appendix~\ref{app:g_cvx_different_metrics}, $\|X\|_1 = \sum_{i,j} |X_{ij}|$ is Euclidean convex but not g-convex:


\begin{figure}[!h]
    \centering
    \includegraphics[width=\linewidth]{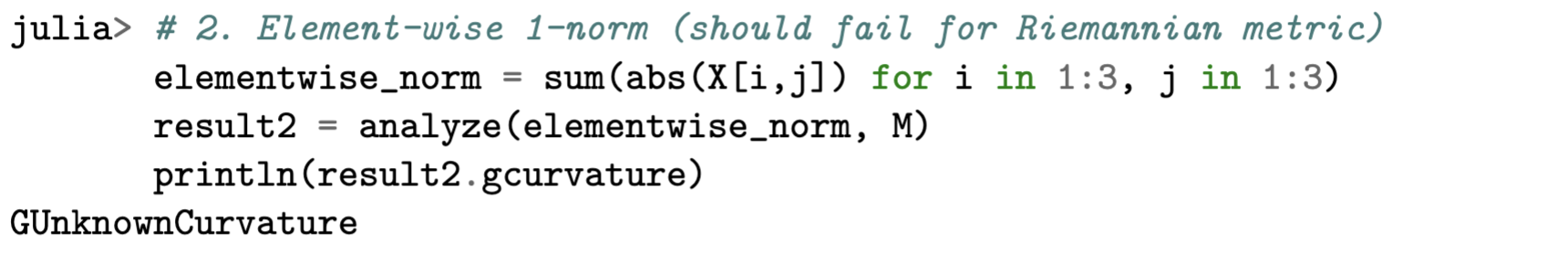}
\end{figure}

\paragraph{Functions Beyond Current Scope.}
DGCP may return \texttt{GUnknownCurvature} for g-convex functions that require atoms not yet in the library:


\begin{figure}[!h]
    \centering
    \includegraphics[width=\linewidth]{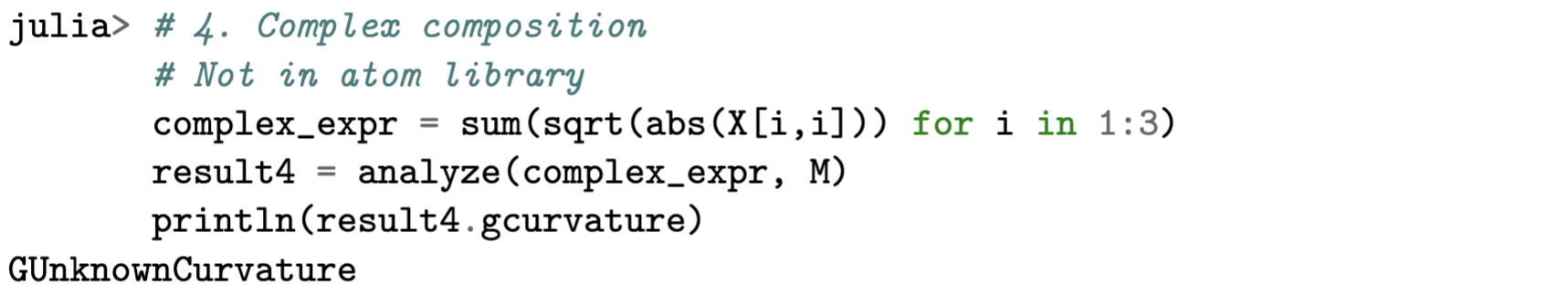}
\end{figure}


\section{Applications}
In this section we illustrate the analysis and verification of geodesic convexity with DGCP on four problems.

\subsection{Matrix Square Root}
Computing the square root $A^{\frac{1}{2}}$ of a symmetric positive definite matrix $A \in \pd$ is an important subroutine in many statistics and machine learning applications. Among other, several first-order approaches have been introduced~\citep{jain2017global,sra2015matrix}. Notably,~\cite{sra2015matrix} gives a geodesically convex formulation of the problem, given by
\begin{equation}\label{eq:sqrt}
    \min _{X \in \pd} \phi(X) := \delta_S^2(X, A)+\delta_S^2(X, I) \; ,
\end{equation}
where $\delta_s$ denotes the s-divergence (Eq.~\ref{eq:sdiv}). Listing~4 
illustrates the use of DGCP to verify the geodesic convexity of Eq.~\ref{eq:sqrt} and leverage the optimization interface to solve the verified problem with a Riemannian solver.

\subsection{Karcher Mean}\label{sec:karcher_mean}

Given a set of symmetric positive definite matrices $\{A_j\} \subseteq \pd$, the Karcher mean is defined as the solution to the problem
\begin{equation}\label{eq:karcher_mean_problem}
    X^* \defas \underset{X \succ 0}{\operatorname{argmin}}\left[\phi(X)=\sum_{i=1}^m w_i \delta_R^2\left(X, A_i\right)\right] \; ,
\end{equation}
where $w_i \geq 0$ are the weights, and 
\begin{equation}
    d^2_R(X, A) = \left \| \log \left(A^{-\frac{1}{2}}X A^{-\frac{1}{2}} \right)\right \|_F^2,
    \qquad X,Y \in \mathbb{P}_d
\end{equation}
is the \textit{Riemannian distance} of the $\mathbb{P}_d$ manifold. The Karcher mean has found applications in medical imaging \citep{Carmichael2013-wq}, kernel methods \citep{clustering}, and interpolation \citep{absil_interpolation}. Since \eqref{eq:karcher_mean_problem} is a conic sum of g-convex functions the problem itself is g-convex. However, the problem is not Euclidean convex. Notably, Problem~\ref{eq:karcher_mean_problem} does not admit a closed form solution for $m>2$. Hence, in contrast to other notions of matrix averages (e.g. arithmetic and geometric mean), the computation of the Karcher mean requires Riemannian solvers.

Using DGCP, we can test and verify these convexity properties as follows:


\begin{figure}[h!]
    \centering
    \includegraphics[width=\linewidth]{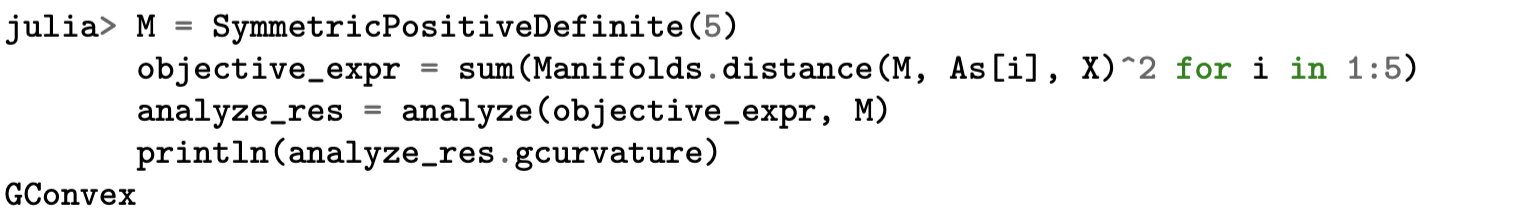}
\end{figure}

\subsection{Computation of Brascamp-Lieb Constants}
The Brascamp-Lieb (short: BL) inequalities~\citep{BL1,BL2} form an important class of inequalities that encompass many well-known inequalities (e.g.  Hölder's inequality,  Loomis–Whitney inequality, etc.) in functional analysis and probability theory. Beyond its applications in various mathematical disciplines, the BL inequalities have applications in machine learning and information theory~\citep{dvir2016rank,pmlr-v30-Hardt13,carlen2009subadditivity,liu2016smoothing}.

Crucial properties of BL inequalities are characterized by so-called  \emph{BL-datum} $(\Ac,w)$, where $\Ac = \big( A_1, \dots, A_m \big)$ is a tuple of surjective, linear transformations and  $\vw=(w_1,\dots,w_m)$ is a vector with real, non-negative entries. The BL datum defines a corresponding BL inequality 
\begin{equation}\label{eq:BL-inequ}
\int_{x \in \R^d} \Bigl( \prod_{j \in [m]} f_j (A_j x)^{w_j} \Bigr) dx 
\leq C(\Ac,\vw) \prod_{j \in [m]} \Bigl( \int_{x \in \R^{d'}} f_j (x) dx	\Bigr)^{w_j} \; ,
\end{equation}
where $f_j: \R^{d'} \rightarrow \R$ denote real-valued, non-negative, Lebesgue-measurable functions. The properties of this inequality a characterized by the \emph{BL-constant}, which corresponds to the smallest constant $C(\Ac,\vw)$ for which the above inequality holds. The value of $C(\Ac,\vw)$ (and whether it is finite or infinite) is of crucial importance in practise. 

The computation of BL constants can be formulated as 
an optimization task on the positive definite matrices~\citep{BL1,BL2}; one formulation of which is given by~\citep{sra_brascamplieb} 
\begin{equation}
\label{eqn:brascamplieb}
    \min_{X \in \pd} \Big[ F(X)=-\log \operatorname{det}(X)+\sum_i w_i \log \operatorname{det}\left(A_i^\top X A_i)\right) \Big] \; .
\end{equation}
This problem is g-convex, but not Euclidean convex, which has motivated the analysis of this problem with g-convex optimization tools~\citep{gurvits,garg2018algorithmic,burgisser2018efficient,thompson}.
We can test and verify the convexity properties of problem~\ref{eqn:brascamplieb} as follows:


\begin{figure}[h!]
    \centering
    \includegraphics[width=\linewidth]{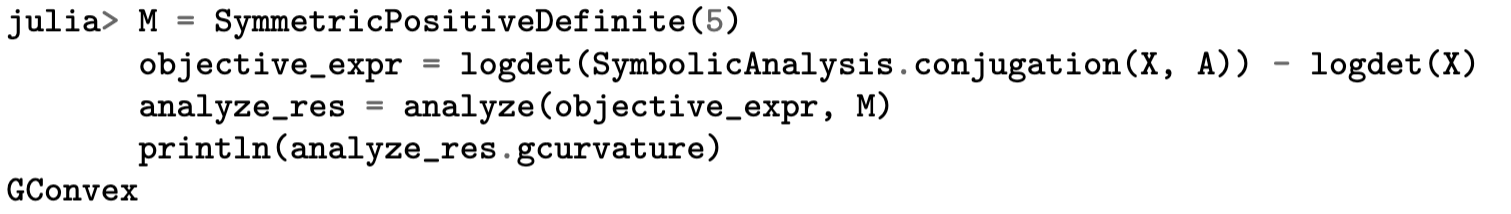}
\end{figure}

\subsection{Robust Subspace Recovery}
Robust subspace recovery seeks to find a low-dimensional subspace in which a (potentially noisy) data set concentrates. Standard dimensionality reduction approaches, such as Principal Component Analysis, can perform poorly in this setting, which motivates the use of other, more robust statistical estimators. One popular choice is Tyler's M-estimator~\citep{Tyler1987}. It can be interpreted as the maximum likelihood estimator for the multivariate student distribution with degrees of freedom parameter $\nu \to 0$~\citep{Maronna2006}. Since the multivariate Student distribution is heavy-tailed, Tyler's M-estimator is more robust to outliers.

Suppose our given data set consists of observations $\{x_i\}_{i=1}^N \subseteq \real_d$. Then Tyler's M-estimator is given by the solution to the following geometric optimization problem, defined on the positive definite matrices:
\begin{equation}\label{eq:Tyler_M}
    \Sigma = \argmin_{\Sigma \in \pd} \frac{1}{n}\sum_{i=1}^n \log \left(x_i^\top \Sigma^{-1}x_i\right) + \frac{1}{d}\log \det \left(\Sigma \right).
\end{equation}
Notably, this problem is g-convex. To see this, note that the function 
\[
f_i(\Sigma) = \log \left(x_i^\top \Sigma^{-1}x_i\right) \qquad i = 1, \ldots n
\]
is g-convex, which follows from the g-convexity of the function $g_i(\Sigma) = \log \left( x_i^\top \Sigma x_i\right)$ (see Proposition~\ref{prop:log_quad_gcvx}) and Lemma~\ref{lemma:inverse_gcvx}. Moreover, the function $f(\Sigma) = \log \det \Sigma$ is g-convex. Thus, problem~\ref{eq:Tyler_M} is g-convex following Proposition~\ref{prop:coniccomb_pwmax}.

 We note that to ensure an unique solution to Problem~\eqref{eq:Tyler_M} one typically enforces the condition $\tr (\Sigma) = c$ for some constant $c > 0$. However, for the purposes of this paper, we restrict our focus on verifying the geodesic convexity of the standard formulation using DGCP; the corresponding expression is shown below.


\begin{figure}[h!]
    \centering
    \includegraphics[width=\linewidth]{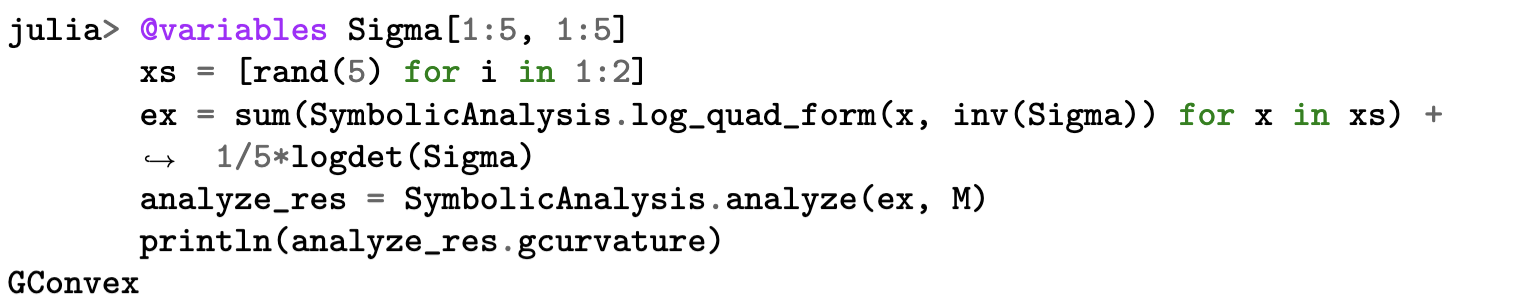}
\end{figure}

\newpage

\subsection{Lorentz Least Squares}

To demonstrate DGCP's versatility across different Cartan-Hadamard manifolds, we consider the least squares problem on the Lorentz model introduced in \ref{sec:lorentzian_atoms}. The least squares problem minimizes the squared error between the data points and a model.

Using DGCP, we can verify whether a given Lorentz least squares problem is geodesically convex by checking if the conditions from \ref{sec:lorentzian_atoms} are satisfied, as demonstrated in the following code snippet:






\begin{figure}[h!]
    \centering
    \includegraphics[width=\linewidth]{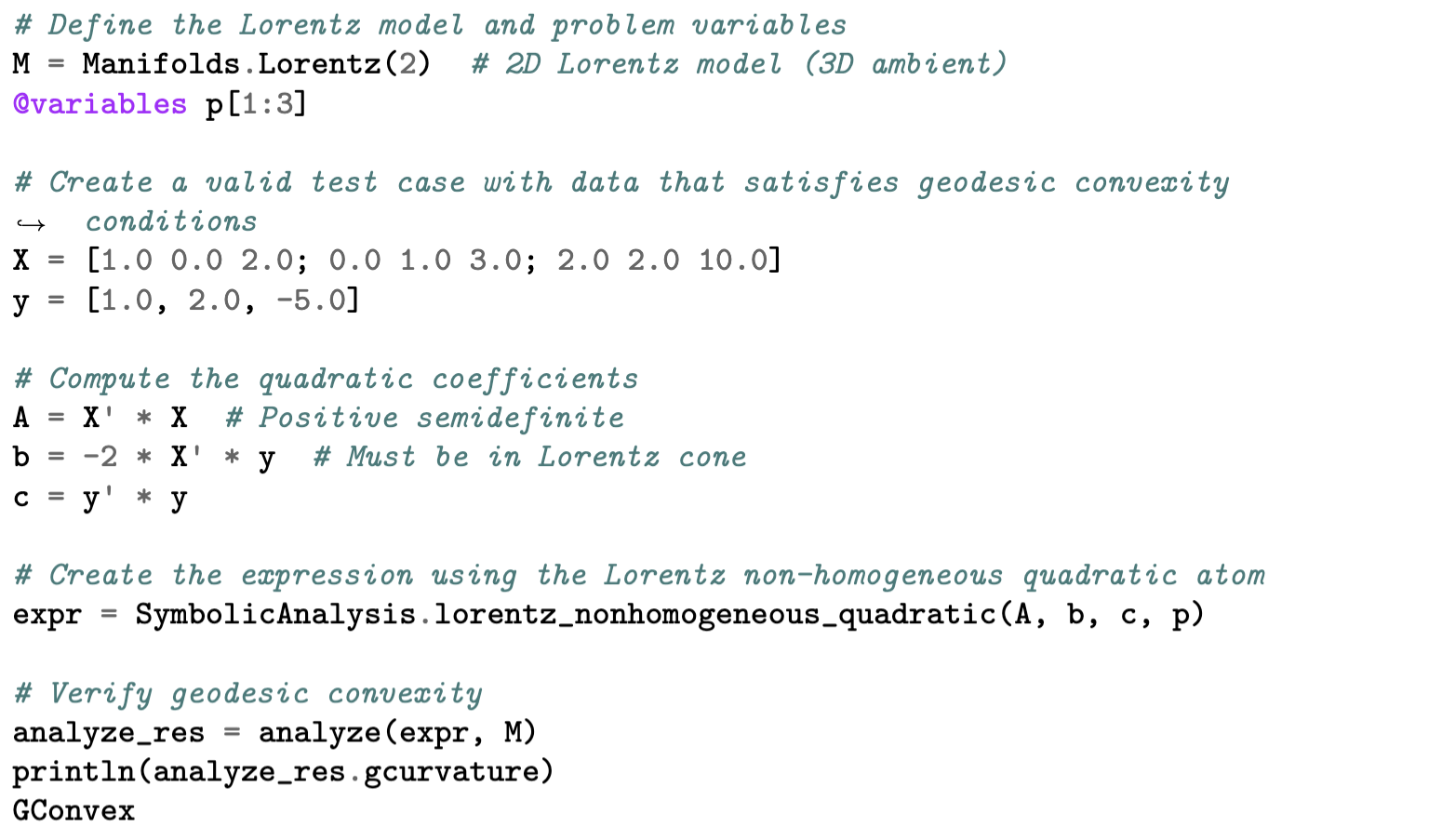}
\end{figure}

This example demonstrates how DGCP extends naturally to different Cartan-Hadamard manifolds beyond the symmetric positive definite matrices, showing the versatility of our framework in verifying geodesic convexity across various geometric settings.
\newpage
\section{Conclusions}
In this paper we introduced the \emph{Disciplined Geodesically Convex Programming} (\emph{DGCP}) framework, which allows for testing and certifying the geodesic convexity of objective functions and constraints in geometric optimization problems. The paper is accompanied by the package \textsl{SymbolicAnalysis.jl}, which implements the foundational atoms and rules of our framework, as well as an interface with \textsl{Manopt.jl} and \textsl{Optimization.jl} that provides access to standard solvers for the verified programs.

The initial implementation of DGCP is limited to basic atoms and rules, which allow for verifying the geodesic convexity of several classical tasks. However, the implementation of additional atoms and rules could significantly widen the range of applications. In particular, future work could focus on implementing additional functionality for verifying program structures that frequently occur in machine learning and statistical data analysis, which we envision as major application areas of our framework. Furthermore, our current framework focuses solely on optimization tasks on symmetric positive definite matrices. While this setting is often considered in the geodesically convex optimization literature, we note that geodesically convex problems arise on more general classes of manifolds, specifically, Cartan-Hadamard manifolds. While we present a general set of rules for geodesic convexity preserving operations on such manifolds, specialized sets of atoms need to be defined for individual manifolds. An extension of the DGCP framework and \textsl{SymbolicAnalysis.jl} package beyond the manifold of symmetric positive definite matrices is an important avenue for future work. Even in the special case of symmetric positive definite matrices, other (Riemannian) metrics could be considered. For instance, recent literature has analyzed optimization tasks on positive definite matrices through the lens of Bures-Wasserstein~\citep{chewi2020gradient} and Thompson~\citep{thompson} geometries. 

The \textsl{Optimization.jl} interface for \textsl{Manopt} is under active development to achieve feature parity. Enhancing this interface will be crucial in enabling the community to more effectively leverage the contributions from this work. Other directions for future work include the improvement and extension of the \textsl{SymbolicAnalysis.jl} package. Currently, we only provide an implementation of DGCP in Julia; however, other languages, in particular Python and Matlab, are popular in the Riemannian optimization community. Hence, providing an implementation in these languages could make our framework more widely applicable.

\newpage
\acks{We thank Shashi Gowda, Christopher Rackauckas, Theo Diamandis, Flemming Holtorf and Alan Edelman from the Julia Lab, as well as Ronny Bergmann, for helpful discussions and comments.

AC and MW were partially supported by the Harvard Dean's Competitive Fund for Promising Scholarship and NSF award CBET-2112085. AC was partially supported by an NSERC Postgraduate Fellowship.

VD is a member of the Julia Lab, which acknowledges the following support: This material is based upon work supported by the National Science Foundation under grant no. OAC-1835443, grant no. SII-2029670, grant no. ECCS-2029670, grant no. OAC-2103804, grant no. DMS-2325184, and grant no. PHY-2021825. The information, data, or work presented herein was funded in part by the Advanced Research Projects Agency-Energy (ARPA-E), U.S. Department of Energy, under Award Number DE-AR0001211 and DE-AR0001222. The views and opinions of authors expressed herein do not necessarily state or reflect those of the United States Government or any agency thereof. This material was supported by The Research Council of Norway and Equinor ASA through Research Council project ``308817 - Digital wells for optimal production and drainage''. Research was sponsored by the United States Air Force Research Laboratory and the United States Air Force Artificial Intelligence Accelerator and was accomplished under Cooperative Agreement Number FA8750-19-2-1000. The views and conclusions contained in this document are those of the authors and should not be interpreted as representing the official policies, either expressed or implied, of the United States Air Force or the U.S. Government. The U.S. Government is authorized to reproduce and distribute reprints for Government purposes notwithstanding any copyright notation herein.
}


\vskip 0.2in
\bibliography{ref}

\newpage

\appendix
\section{Deferred Proofs}
\label{app:theorem}
\paragraph{Notation.} For any two symmetric positive definite matrices $A, B \in \pd$ we use the notation $A \sharp B$ to denote the geometric mean between $A$ and $B$
\[
A \sharp B \defas A^{\frac{1}{2}}\left(A^{-1/2} B A^{-1/2}\right)^\frac{1}{2}A^{\frac{1}{2}}.
\]
Moreover, we use the $A \sharp_t B$ to denote the geodesic connecting $A$ to $B$
\[
A \sharp_t B \defas A^{\frac{1}{2}}\left(A^{-1/2} B A^{-1/2}\right)^t A^{\frac{1}{2}} \qquad \forall t \in [0,1].
\]

We will use the following lemma in the proofs to come. 

\begin{lemma}\label{lemma:inv_commute_sharp}
    For any $A, B \in \pd$ it holds that
    \[
    \left(A \sharp_t B\right)^{-1} = A^{-1} \sharp_t B^{-1} \; .
    \] 
\end{lemma}
\begin{proof}
    This follows from the basic computation
    \begin{align*}
        \left(A \sharp_t B\right)^{-1} &= \left(A^{\frac{1}{2}}\left(A^{-1/2} B A^{-1/2}\right)^t A^{\frac{1}{2}} \right)^{-1} \\
        &= A^{-\frac{1}{2}}\left(A^{1/2} B^{-1} A^{1/2}\right)^t A^{-\frac{1}{2}}\\
        &= A^{-1} \sharp_t B^{-1} \;.
    \end{align*}
\end{proof}

\begin{lemma}[Midpoint convexity]
     A continuous function $f$ on a g-convex set $\mathcal{S}\subseteq \mathcal{M}$  is g-convex if $f\left(X \sharp Y \right) \leq \frac{1}{2} f\left(X\right)+\frac{1}{2} f\left(Y\right)$ for any $X,Y  \in \mathcal{S}$.
\end{lemma}
\begin{proof}
    The proof is analogous to showing the Euclidean midpoint convex condition. Namely, instead of recursively applying the hypothesis to line segments of length $2^{-k}$ for $k \in \nat$, we apply it to the midpoints of geodesic segments.

    Let $X_0 ,Y_0 \in \mathcal{S}$. Let $\gamma:[0,1] \to \mathcal{M}$ be a geodesic segment such that $\gamma(0)=X_0 \neq Y_0 = \gamma(1)$ and $\gamma(t) \in S$ for all $t \in [0,1]$. 

    We need to verify $f$ is geodesically convex, i.e. show that 
    \begin{equation}\label{eq:f_gcvx}
    f(\gamma(t)) \leq (1-t)f(\gamma(0)) + t f(\gamma(1))    
    \end{equation}
    holds for all $t \in [0,1]$.
    The hypothesis implies \eqref{eq:f_gcvx} holds for $t = \frac{1}{2}$. Since $\gamma(\frac{1}{2}) \in \mathcal{S}$, we can now recursively apply the hypothesis to the sub-geodesic segments defined by the images $\gamma\left([0,\frac{1}{2}]\right)$ and $\gamma\left([\frac{1}{2}, 1]\right)$. In turn, \eqref{eq:f_gcvx} holds for $t \in \{0, \frac{1}{4}, \frac{1}{2}, \frac{3}{4}, 1\}$. Applying this argument $k$ times shows that \eqref{eq:f_gcvx} holds for $t \in \mathcal{I}_K \defas \{\frac{\ell}{2^k}: 0 \leq \ell \leq 2^k \}$. The set $\mathcal{I}_\infty$ is dense in $[0,1]$ the argument follows by the continuity of $f$. 
\end{proof}

\subsection{Rules}\label{app:gcvx_rules}
\begin{proof}[Proposition~\ref{prop:coniccomb_pwmax}]

We prove the proposition for the case $n=2$ and note that the arguments can be easily generalized for arbitrary $n \in \mathbb{N}$. 

Consider $f,g: S \subseteq \mathcal{M} \to \real$ to be two g-convex functions on a g-convex set $S$. Let $x,y \in S$ and $\gamma:[0,1] \to \mathcal{M}$ be a geodesic that connects $\gamma(0) = x$ to $\gamma(1) = y$ such that $\gamma[0,1] \subseteq S$. Then for all $t \in [0,1]$,
\begin{align*}
    \alpha f(\gamma(t)) + \beta g(\gamma(t)) &\leq \alpha \bigg( (1-t)f(\gamma(0)) + t f(\gamma(1))\bigg) + \beta \bigg( (1-t)g(\gamma(0)) + t g(\gamma(1))\bigg)
    \\&= \left(1-t\right)\big(\alpha f(\gamma(0)) + \beta g(\gamma(0)) \big) + t \big(\alpha f(\gamma(1)) + \beta g(\gamma(1))\big).
\end{align*}
Moreover, 
\begin{align*}
    \max \big\{f(\gamma(t)), g(\gamma(t)) \big\} & \leq \max \big\{(1-t)f(\gamma(0)) + t f(\gamma(1)), (1-t)g(\gamma(0)) + t g(\gamma(1))  \big\}
    \\&\leq (1-t) \max\big \{f(\gamma(0)),g(\gamma(0))  \big\} + t \max\big\{f(\gamma(1)), g(\gamma(1))\big \}.
\end{align*}
\end{proof}

\begin{proof}[Proposition~\ref{prop:ecvx_composition}]
By applying convexity results and the fact that $h(\cdot)$ is nondecreasing we obtain
\[
h(f(\gamma(t)) \leq h \left((1-t)f(\gamma(0)) + t f(\gamma(1))\right) \leq (1-t)h(f(\gamma(0))) + t h(f(\gamma(1))).
\]
\end{proof}

\begin{proof}[Proposition~\ref{lemma:inverse_gcvx}]
    Suppose $A, B \in \pd$ and $f(X)$ is g-convex. Then for all $t \in [0,1]$ we have 
    \[
    g(A \sharp_t B) = f\left( (A \sharp_t B)^{-1}\right) = f(A^{-1} \sharp_t B^{-1}) \leq \left(1-t\right)f\left(A^{-1}\right) + t f\left(B^{-1}\right) = (1-t)g(A) + tg(B)
    \]
    where in the second equality we applied Lemma~\ref{lemma:inv_commute_sharp}.
    
\end{proof}

In order to prove Proposition~\ref{prop:gcvx_affine_positive}
we need the following lemmas.

\begin{lemma}[Theorem 4.1.3 \cite{bhatia07positivedefinitematrices}]\label{lemma:extremal_characterization}
        Let $A, B \in \pd$. Their geometric mean $A \sharp B$ satisfies the following extremal property: 
        \[
        A \sharp B = \max \{X: X = X^\top, \ \begin{bmatrix}
            A &X \\
            X &B
        \end{bmatrix} \succeq 0\}.
        \]
        In particular, if $X$ is symmetric and satisfies the condition 
        \[
        \begin{bmatrix}
            A &X \\
            X &B
        \end{bmatrix} \succeq 0
        \]
        then $A \sharp B \succeq X$.
    \end{lemma}

\begin{lemma}\label{lemma:psd_block_matrix}
        If $X \succeq 0$ then the matrix $\tilde{X}$ defined as follows satisfies 
        \[
        \tilde{X} = \begin{bmatrix}
            X & X  \\
            X & X
        \end{bmatrix} \succeq 0.
        \]
    \end{lemma}

\begin{lemma}\label{positive_linear_gm}
        Let $B \succeq 0$ and $\Phi(X)$ be a positive linear map, that is, $\Phi(X) \succeq 0 $ whenever $X \succeq 0$.Then the function $\phi: \real^{d \times d} \to \real^{d \times d}$ defined by $\phi(X) \defas \Phi(X) + B$
        \[
        \phi\left(X \sharp Y\right) \preceq \phi(X) \sharp \phi(Y) \qquad \forall X,Y \in \pd.
        \]
    \end{lemma}

\begin{proof}[Lemma~\ref{positive_linear_gm}]
        Let $X, Y \in \pd$ and since $X \sharp Y \in \pd$ we have by Exercise 3.2.2 (ii)~\cite{bhatia07positivedefinitematrices} that 

         \begin{equation}\label{eq:Phi_succ_0}
        \begin{aligned}
        \begin{bmatrix}
            X  & X \sharp Y  
            \\ X\sharp Y   &Y 
        \end{bmatrix} \succeq 0 \implies 
\begin{bmatrix}
            \Phi(X)  & \Phi\left(X \sharp Y\right)  
            \\ \Phi\left(X\sharp Y\right)   &\Phi(Y) 
        \end{bmatrix}
        \succeq 0 .
        \end{aligned}
    \end{equation}
By applying Lemma~\ref{lemma:psd_block_matrix} we have
\[
\begin{bmatrix}
            B & B \\
            B & B
        \end{bmatrix} \succeq 0
\]
thus we have

         \begin{equation}
        \begin{aligned}
\begin{bmatrix}
            \Phi(X)  & \Phi\left(X \sharp Y\right)  
            \\ \Phi\left(X\sharp Y\right)   &\Phi(Y) 
        \end{bmatrix} 
        + 
        \begin{bmatrix}
            B & B \\
            B & B
        \end{bmatrix}
        =
        \begin{bmatrix}
            \Phi(X)+ B &\Phi\left(X \sharp Y\right)   + B \\
            \Phi\left(X\sharp Y\right)+ B & \Phi(Y) + B
        \end{bmatrix}
        = \begin{bmatrix}
            \phi(X)  & \phi\left(X \sharp Y\right)  
            \\ \phi\left(X\sharp Y\right)   &\phi(Y) 
        \end{bmatrix} 
        \succeq 0 .
        \end{aligned}
    \end{equation}

 By applying the extremal characterization of geometric mean we get $\phi(X) \sharp \phi(Y) \succeq \phi(X \sharp Y)$ which is our desired result.
\end{proof}

Now we can prove Proposition~\ref{prop:gcvx_affine_positive}.

\begin{proof}[Proposition~\ref{prop:gcvx_affine_positive}]
    It suffices to check midpoint convexity. 
    \[
    \begin{aligned}
        g(X \sharp Y) &\defas f\left(\phi(X\sharp Y) \right) 
        \\& \preceq f \left( \phi(X)  \sharp \phi(Y) \right) \qquad (\text{Lemma~\ref{positive_linear_gm})}
        \\& \preceq \frac{f(\phi(X)) + f(\phi(Y))}{2} \qquad (f \text{ is g-convex})
        \\& = \frac{g(X) + g(Y)}{2}.
    \end{aligned}
    \]
\end{proof}

\subsection{Atoms}\label{app:gcvx_atoms}
\subsubsection{SPD Atoms.}
In this section, we prove that the list of atoms in Section ~\ref{sec:atoms} is g-convex with respect to the canonical Riemannian metric. The proofs demonstrate the application of the propositions found in Section~\ref{sec:rules}.

\begin{lemma}[Epigraphs and g-convexity (Lemma 2.2.1,~\citet{bacak2014convex})]\label{lemma:epigraph_gvx}
    Let $f:\pd \to \real$ be geodesically convex and define its epigraph as $$\epi(f) \defas \{(X,t) : X \in \pd \text{ and } f(X) \leq t\} \subseteq S \times \real.$$ Then $f$ is geodesically convex if and only if $\epi(f)$ is a closed geodesically convex subset of $\pd \times \real$.
\end{lemma}

\begin{prop}\label{prop:sup_gvx}
 Let $S\subseteq \real^d$ and  $y \in S$. Suppose $f(X,y): \pd \to \real$ is g-convex in $X$, then define the function $g: \pd \to \real$ by 
    \begin{equation*}
        g(X) = \sup_{y \in S}f(X,y).
    \end{equation*}
    Then $g(X,y)$ is g-convex on $\pd$ with respect to the canonical Riemannian metric. The domain of $g$ is 
    \begin{equation*}
        \dom (g) = \{X \in \pd : (X,y) \in \dom(f) \text{ for all } y \in S, \ \sup_{y \in S}f(X,y) < \infty\}.
    \end{equation*}   
\end{prop}

\begin{proof}
     We claim that 
    \begin{equation*}
        \epi(g) = \bigcap_{y \in S}\epi(f(\cdot, y)) \defas \bigcap_{y \in S}\{(X,t): f(X,y) \leq t\}.
    \end{equation*}
    Let $(X,t) \in \epi(g)$. Then 
    \begin{align*}
        \begin{split}
            &\sup_{y \in S} f(X,y) \leq t \text{ and } X \in \dom(f) 
            \\&\iff f(X,y) \leq t \text{ for all } y \in S  \text{ and } X \in \dom(f) 
            \\&\iff (X,t) \in \bigcap_{y \in S} \epi(f)(\cdot, y).
        \end{split}
    \end{align*} 
But $f(\cdot, y)$ is g-convex hence $\epi f(\cdot, y)$ is g-convex for all $y \in S$. Now note that the intersection of g-convex sets on Cartan-Hadamard manifolds (e.g., $\pd$) is g-convex (see Chapter 11 in~\citet{boumal2020introduction}). By Proposition~\ref{lemma:epigraph_gvx} we obtain our desired result. 
\end{proof}

\begin{prop}
Let $h, h_1 \ldots, h_n \in \real^d$ be fixed. The following functions $f:\pd \to \real$ are geodesically convex with respect to the canonical Riemannian metric. 
    \begin{enumerate}[label=(\theenumi)]
        \item $f(X) = \log \left(\sum_{i=1}^n h_i^\top X h_i\right)$
        \item $f(X) = \log \det(X)$
        \item $f(X) = h^\top X h$
        \item $f(X) = \tr (X)$
        \item $f(X) = \delta_S^2(X,Y):= \log \det \left(\frac{X+Y}{2}\right) - \frac{1}{2}\log\det(XY)$ for fixed $Y \in \pd$.
        \item $f(X,Y) = \|\log \left(Y^{-\frac{1}{2}}X Y^{-\frac{1}{2}} \right) \|_F^2$ for fixed $Y \in \pd$.
        \item $f(X) = \sup_{\{y:\real^d : \|y\|_2=1\}}y^\top X y$
        \item $f(X) = X^{-1}$.
    \end{enumerate}
\end{prop}

\begin{proof}
    We defer the proofs of (1), (2), and (3) to Propositions~\ref{prop:log_quad_gcvx}, \ref{prop:prove_logdet_gcvx}, and \ref{prop:quad_gcvx} respectively.
    \paragraph{(4)} It is clear that $\tr(X)$ is a strictly positive linear map and thus by Proposition~\ref{prop:strict_positive_linear} it is g-convex.
    \paragraph{(5)} For the S-divergence, we apply Proposition~\ref{prop:sra_thm15} with $h_1(X) = \log \det(X)$ and $\Phi(X) = \frac{X+Y}{2}$, i.e., the function 
    \[
    h_1(\Phi(X)) = \log \det \left(\frac{X+Y}{2}\right)
    \]
    is g-convex. Moreover, by Proposition~\ref{prop:logdet_gcvx}, we have that 
    \[
    X \mapsto -\log \det(X)
    \]
    is g-convex (in fact, g-linear) and so 
    \[
    h_2(X) = -\frac{1}{2}\log\det(XY) = - \frac{1}{2}\left(\log \det(X) + \log \det(Y)\right)
    \]
    is g-convex. Since conic combinations of g-convex functions are g-convex (see Proposition~\ref{prop:coniccomb_pwmax}) we have that
    \[
    \delta_S^2(X,Y) = h_1(X) + h_2(X) 
    \]
    is g-convex.
    \paragraph{(6)} We refer the reader to Corollary~19~ (\citep{sra2015conic}) for a proof involving symmetric gauge functions. For a more general proof we refer the reader to Corollary~6.1.11~(\citep{bhatia_psd}).
    \paragraph{(7)} This is a direct consequence of Proposition~\ref{prop:sup_gvx}.
    \paragraph{(8)} It suffices to establish midpoint convexity. Observe that for any $A,B \in \pd$ 
    \[
    \left(A \sharp B\right)^{-1} = \left( A^{\frac{1}{2}} \left(A^{-\frac{1}{2}} B A^{-\frac{1}{2}} \right)^t A^{\frac{1}{2}}  \right)^{-1} = A^{-\frac{1}{2}} \left(A^{\frac{1}{2}} B^{-1} A^{\frac{1}{2}} \right)^t A^{-\frac{1}{2}} = A^{-1} \sharp B^{-1}.
    \]
    It follows from the AM-GM inequality for positive linear operators that 
    \[
    A^{-1} \sharp B^{-1} \preceq \frac{A^{-1} + B^{-1}}{2} \; ,
    \]
    thus verifying g-convexity. 
\end{proof}

\subsubsection{Lorentzian Atoms.}
The g-convexity of the atoms in Section~\ref{sec:lorentzian_atoms} are proven in \cite{Ferreira2022} and \cite{Ferreira2023_nonhomogeneous}. 




\begin{theorem}[\citep{Ferreira2022}]\label{theorem:special_cases_hom}
Let $A \in \mathbb{R}^{(d+1) \times(d+1)}$ and $f: \mathbb{H}^d \rightarrow \mathbb{R}$ be defined by $f(p)=p^{\top} A p$. Then

\begin{enumerate}
\item  If $\sigma \geq-\lambda_{\min }(\bar{A})$ and $a=0$, then $f$ is geodesically convex;
\item  If $\sigma+\lambda_{\min }(\bar{A})>2 \sqrt{a^{\top} a}$, then $f$ is geodesically convex.
\end{enumerate}

\end{theorem}




\begin{proposition}
    In the following we prove the atoms in the Section~\ref{sec:lorentzian_atoms} are geodesically convex.
\end{proposition}
\begin{proof}
\begin{enumerate}
    \item \textbf{Lorentzian Distance. }$(\mathbb{H}_d, d_\Lorentz)$ is a Cartan-Hadamard manifold where $d_\Lorentz$ is its intrinsic distance. Since all intrinsic distances of Cartan-Hadamard manifolds is g-convex then $d_\Lorentz$ is g-convex.
    \item \textbf{Log-Barrier.}  \citep{Ferreira2022} applies the second-order condition of g-convexity to prove the result.
    \item \textbf{Homogeneous SPD.} See \citep{Ferreira2022}.
    \item \textbf{Nonhomogeneous SPD.} G-convexity directly follows from applying Theorem~\ref{theorem:special_cases_hom} (1).
    \item \textbf{Least Squares.} Since $A^\top A $ is symmetric positive semidefinite we know that the homogeneous function $h(p) = p^\top A^\top A p$ is a geodesically convex atom. One way to prove this problem is geodesically convex is to invoke Proposition~\ref{prop:nonhom_hom} and check the condition $-b^\top A \in \mathscr{L}$, or equivalently, check the inequality
\[ \left(2A^\top b\right)_{d+1} \leq - \sqrt{2}\| \overline{A^\top b}\|_2.
\]
\end{enumerate}
\end{proof}


\section{Additional results and discussion of g-convexity}\label{app:g_cvx_different_metrics}
We show that geodesic convexity, like Euclidean convexity, is generally not preserved under products. 
\paragraph{Counterexample.}
For simplicity and without loss of generality we take $\log(\cdot) := \log_2(\cdot)$. We take $A = \operatorname{Diag}(1,1)$ and $B := \operatorname{Diag}(16, 16)$ and the two g-convex functions to be $f_1(X):= \operatorname{tr}(X)$ and $f_2(X) = - \log \det (X)$. We show that $(f_1 f_2)(X) := -\tr(X) \log \det(X)$ is not g-convex. To this end, suppose $t=1/2$. Then 
\[
\begin{aligned}
    \gamma(1/2) := A^{1/2}\left(A^{-1/2} B A^{-1/2}\right)^t A^{1/2} = \operatorname{Diag}(4, 4).
\end{aligned}
\]
Thus $f_1(\gamma(1/2)) f_2(\gamma(1/2)) = - 32.$ Moreover, observe that 
\[
\begin{aligned}
    f_1(A) = 2, \qquad & f_2(A) = 0 
    \\f_1(B) = 32, \qquad & f_2(B) = -8.
\end{aligned}
\]
Finally, we obtain 
\[
\frac{1}{2}\left(f_1(A) f_2(A) \right) + \frac{1}{2}\left(f_1(B) f_2(B) \right) = -128
\]
Thus 
\[
f_1(\gamma(1/2)) f_2(\gamma(1/2)) > \frac{1}{2}\left(f_1(A) f_2(A) \right) + \frac{1}{2}\left(f_1(B) f_2(B)\right) 
\]
thus $(f_1 f_2)(X)$ is not g-convex.
\hfill $\square$

We show a function that is g-convex with respect to the Euclidean metric but not with respect to the canonical Riemannian metric.

\begin{prop}[\citet{example-bien}]
    The function  $f(X) := \|X\|_1 := \sum_{i,j} |X_{ij}|$ is g-convex with respect to the Euclidean metric but not with respect to the canonical Riemannian metric.
\end{prop}

\begin{proof}
    Let $f(X) := \|X\|_1 := \sum_{i,j} |X_{ij}|$ be the element-wise 1-norm. Observe for all $X,Y \in \pd$
\[
f\left(\theta X + (1-\theta) Y\right) = \sum_{ij=1}^d\left | \theta X_{ij} + (1-\theta) Y_{ij}\right| \leq \theta \sum_{ij=1}^d |X_{ij}| + (1-\theta)\sum_{ij=1}^d |Y_{ij}| = \theta f(X) + (1-\theta)f(Y) .
\]
This establishes that $f$ is g-convex with respect to the Euclidean metric on $\pd$.
In contrast, take the matrices 
\[
\Sigma_1=I_3 \qquad \text{and} \qquad \Sigma_2=\left(\begin{array}{ccc}
1.0 & 0.5 & -0.6 \\
0.5 & 1.2 & 0.4 \\
-0.6 & 0.4 & 1.0
\end{array}\right).
\]
Let $\gamma:[0,1] \to \pd$ be the geodesic induced by the canonical Riemannian.  metric. That is,
\[
\gamma(t) = \Sigma_1^{1/2}\left(\Sigma_1^{-1/2}\Sigma_2 \Sigma_1^{-1/2}\right)^t \Sigma_1^{1/2}.
\]
Then observe that 
\[
f(\gamma(1/2)) = \|\Sigma_2^{1/2}\|_1 = 4.7638... >  4.6 = \frac{1}{2}\|\Sigma_1\|_1 + \frac{1}{2}\|\Sigma_2\|_1 = \frac{1}{2}f(\Sigma_1) + \frac{1}{2}f(\Sigma_2)
\]
which violates the definition of g-convex of $f$. 
\end{proof}

The following two examples are g-convex with respect to the canonical Riemannian metric but not with respect to the Euclidean metric.

\begin{prop}\label{prop:log_quad_gcvx}
    Let $y_i \in \real^d$ be nonzero vectors for $i = 1, \ldots, n$. The function 
    \[
    f(X) = \log \left(\sum_{i=1}^n y_i^\top X y_i  \right)
    \]
    is g-convex with respect to the canonical Riemannian metric but is not g-convex with respect to the Euclidean metric.
\end{prop}
\begin{proof}
    First we show that $f(X)$ is not g-convex with respect to the Euclidean metric. Observe that for any $y \in \real^d \setminus \{0\}$,  $\theta \in (0,1)$ and $X, Y \in \pd$, we have
    \[
    \begin{aligned}
        \log \left(y^\top \left(\theta X + (1-\theta)Y\right) y \right) &= \log \left(\theta y^\top X y + (1-\theta)y^\top Y y\right) 
        \\&> \theta \log \left(y^\top X y\right) + (1- \theta) \log \left(y^\top Y y \right)
    \end{aligned}
    \]
    where the strict inequality follows from the fact that $\log(\cdot)$ is a strict concave function on $(0, \infty)$. 

    To prove that $f(X)$ is g-convex with respect to the canonical Riemannian metric, we follow the proof from Lemma 1.20~\citep{wieselstructuredcovariance} and Lemma 3.1~\citep{zhang2016robust}. To this end, let $X,Y \in \pd$ and verify the midpoint convexity condition
    \[
    f(X \sharp Y) \leq \frac{1}{2}f(X) + \frac{1}{2}f(Y)
    \]
    where $\sharp$ denotes the geometric mean of $X$ and $Y$. By simple algebra one can show that the condition above is equivalent to 
    \begin{equation}\label{eq:g_cvx_logquad}
        \left(\sum_{i=1}^n {y}_i^T[X \sharp Y] {y}_i\right)^2 \leq\left(\sum_{i=1}^n y_i^T X {y}_i\right)\left(\sum_{i=1}^n y_i^T Y {y}_i\right).
    \end{equation}
    
    For simplicity, we define 
    \[
    u_i := X^{\frac{1}{2}}y_i \qquad \text{and} \qquad v_i := \left(X^{-\frac{1}{2}}Y X^{-\frac{1}{2}}\right)^{\frac{1}{2}}X^{\frac{1}{2}}y_i.
    \]
    Observe that by applying Cauchy-Scwartz twice we get
    \[
    \begin{aligned}
\left(\sum_{i=1}^n \mathbf{u}_i^T \mathbf{v}_i\right)^2 & =\left(\sum_{i=1}^n\left|\mathbf{u}_i^T \mathbf{v}_i\right|\right)^2 \\
& \leq\left(\sum_{i=1}^n\left\|\mathbf{u}_i\right\|\left\|\mathbf{v}_i\right\|\right)^2 \\
& \leq\left(\sum_{i=1}^n\left\|\mathbf{u}_i\right\|^2\right)\left(\sum_{i=1}^n\left\|\mathbf{v}_i\right\|^2\right).
\end{aligned}
\]

It suffices to check that 
\[
\left(\sum_{i=1}^n \mathbf{u}_i^T \mathbf{v}_i\right)^2 \leq\left(\sum_{i=1}^n\left\|\mathbf{u}_i\right\|^2\right)\left(\sum_{i=1}^n\left\|\mathbf{v}_i\right\|^2\right)
\]
if and only if \eqref{eq:g_cvx_logquad} holds.
\end{proof}

\begin{prop}\label{prop:prove_logdet_gcvx}
    The function  $f(X) = \log \det X$ is g-convex (in fact, g-linear) with respect to the canonical metric but is g-concave with respect to the Euclidean metric.
\end{prop}
\begin{proof}
    To show that $f:\pd \to \real_{++}$ is indeed g-concave with respect to  
    the Euclidean metric we refer the reader to Section 3.1.5~\citep{Boyd_Vandenberghe_2004}. Let $X, Y \in \pd$ and $\gamma:[0,1] \to \pd$ be the geodesic segment connecting $\gamma(0) = A$ to $\gamma(1) = B$. For $t \in [0,1]$
    \[
    \begin{aligned}
         \log \det \left(\gamma(t) \right) &= \log \det \left(X^{1/2}(X^{-1/2}YX^{-1/2})^t X^{1/2} \right) 
            \\&= \log \left( \det(X) \det(X^{-1})^t \det(Y)^t \right) 
            \\&= \log \det(X) - t \log \det(X) + t \log \det(Y) 
            \\&= (1-t) \log \det (X) + t \log \det (Y).
    \end{aligned}
    \]
\end{proof}
Finally, we show an example of a function that is g-convex with respect to both the Euclidean and canonical Riemannian metric. To this end, we need the following lemma.

\begin{lemma}\label{lemma:diagonalize_pd}[Theorem 7.6(a)~\citep{horn_matrixanalysis}]
        Let $A, B \in \pd$ be two positive definite matrices. Then $A$ and $B$ are simultaneously diagonalizable by a congruence, i.e., there exists a nonsingular matrix $S \in \real^{n \times n}$ such that 
        \[
        A = S I S^\top \qquad \text{and} \qquad B = S \Lambda S^\top
        \]
        where the main diagonal entries of $\Lambda$ are the eigenvalues of the diagonal matrix $A^{-1} B$. In fact, one possible choice of $S$ is $S = A^{\frac{1}{2}} U$ where $U$ is any orthogonal matrix such that $A^{-\frac{1}{2}}B A^{-\frac{1}{2}} = U \Lambda U^\top$ is a spectral decomposition.
    \end{lemma}

\begin{prop}\label{prop:quad_gcvx}
    Fix $y \in \real^d\setminus\{0\}$. The function $f(X) = y^\top X y$ is g-convex with respect to both the Euclidean metric and the canonical Riemannian metric.
\end{prop}
\begin{proof}
    We can apply the \textit{trace trick} to write
    \[
    f(X) = y^\top X y = \tr\left(X y y^\top\right) = \tr\left(X Y\right)
    \]
    where $Y \defas y y^\top$. With respect to the Euclidean metric, we observe that $f(X)$ is a composition of g-linear functions and thus it is g-linear with respect to the Euclidean metric. That is, for all $\theta \in [0,1]$ and $X,Z \in \pd$ we have 
    \[
    f(\theta X + (1-\theta)Y) = \tr\left( \left(\theta X + (1-\theta)Z\right)Y\right) = \theta \tr\left(XY\right) + (1-\theta)\tr \left(ZY\right) = \theta f(X) + (1-\theta) f(Z).
    \]

    Now we show $f(X)$ is g-convex with respect to the canonical Riemannian metric. Apply Lemma~\ref{lemma:diagonalize_pd} to obtain 

    \[
     A = S I S^\top \qquad \text{and} \qquad B = S \Lambda S^\top
    \]

    where we choose $S = A^{\frac{1}{2}} U$ where $U$ is any orthogonal matrix such that $A^{-\frac{1}{2}}B A^{-\frac{1}{2}} = U \Lambda U^\top$ is a spectral decomposition. Then the geodesic that connects $A$ to $B$ is reduced as follows:

    \begin{align*}
    \begin{split}
        \gamma(t)  &= A^{\frac{1}{2}}\left(A^{-\frac{1}{2}}B A^{-\frac{1}{2}}\right)^tA^{\frac{1}{2}}
        \\ &= A^{\frac{1}{2}} \left(U \Lambda U^\top \right)^t A^{\frac{1}{2}}
        \\ &= A^{\frac{1}{2}} U \Lambda^t U^\top A^{\frac{1}{2}}. 
    \end{split}
\end{align*}

Hence for $t \in [0,1]$ we have 
\begin{equation*}
    \phi(\gamma(t)) = \left( y^\top  A^{\frac{1}{2}} U  \right) \Lambda^t \left(U^\top A^{\frac{1}{2}}y \right) = \tilde{y}^\top \Lambda^t \tilde{y}
\end{equation*}
where $\tilde{y} = U^\top A^{\frac{1}{2}}y$. Since $U$ orthogonal and $A^{\frac{1}{2}} \in \pd$ we have that $U^\top A^{\frac{1}{2}}$ is invertible and thus acts as a change-of-basis that diagonalizes the quadratic form $\phi(\gamma(t))$. In fact, the eigenvalues of such a diagonalization are precisely the generalized eigenvalues of the pair matrices $(B, A)$ raised to the $t$-th power.

Also, we have
\begin{align*}
    \begin{split}
        (1-t)\phi(A) + t \phi(B) & = y^\top \left((1-t)A + t B\right) y 
        \\ &= y^\top \left( (1-t)S S^\top + t S \Lambda S^\top \right)y 
        \\&= y^\top S \left((1-t)I + t \Lambda  \right)S^\top y
        \\&= \left(y^\top A^{\frac{1}{2}}U \right)\left((1-t)I + t \Lambda  \right) \left(U^\top A^{\frac{1}{2}}y \right)
        \\&= \tilde{y}^\top \left((1-t)I + t \Lambda\right) \tilde{y}.
    \end{split}
\end{align*}
Finally, $\phi$ is geodesically convex if and only if 
\[
\phi(\gamma(t)) = \Tilde{y}^\top \Lambda^t \Tilde{y} \leq \tilde{y}^\top \left((1-t)I + t \Lambda\right) \tilde{y} = (1-t)\phi(A) + t \phi(B) \qquad \forall t \in [0,1].
\]
Since $\Lambda^t$ and $(1-t)I + t \Lambda$ are both diagonal matrices we have the equivalent inequality
\[
\Lambda^t \defas \diag(\lambda_1^t, \ldots, \lambda_n^t) \preceq (1-t)I + t \Lambda \qquad \forall t \in [0,1].
\]
By the weighted AM-GM inequality, we indeed have 
\[
\lambda_i^t \leq (1-t) + t \lambda_i \qquad \forall i \in [n] \ \forall t \in [0,1]. 
\]
Since $y \in \real^n$ was arbitrarily selected and we proved 
\[
\phi(\gamma(t)) \leq (1-t)\phi(A) + t \phi(B) \qquad \forall t \in [0,1]
\]
our desired result is proved.
\end{proof}
\[\]

\end{document}